\pdfoutput=1  

\documentclass[11pt]{article}
\usepackage[margin=.8in,left=.8in]{geometry}
\usepackage{amsthm}
\usepackage{amssymb}
\usepackage{amsmath}
\usepackage{mathtools}
\usepackage[]{xcolor}
\usepackage{colortbl}    
\usepackage{stmaryrd}    
\usepackage[safe]{tipa} 
\usepackage{adjustbox} 
\usepackage{makecell}
\usepackage{fontawesome}

\usepackage{hyperref}
\hypersetup{
  colorlinks = true,
  allcolors=.           
}

\usepackage[capitalise]{cleveref}

\crefformat{section}{\S#2#1#3} 
\crefformat{subsection}{\S#2#1#3}
\crefformat{subsubsection}{\S#2#1#3}

\usepackage{tikz}
\usetikzlibrary{cd}

\DeclareMathAlphabet{\mathpzc}{OT1}{pzc}{m}{it} 
\newcommand\mathscr[1]{\scalebox{1.1}{$\mathpzc{#1}$}}

\definecolor{darkblue}{rgb}{0.05,0.25,0.65}
\definecolor{greenii}{RGB}{20,140,10}
\definecolor{lightgray}{rgb}{0.9,0.9,0.9}
\definecolor{orangeii}{RGB}{200,100,5}

\usepackage{multirow}

\usepackage{enumitem}

\makeatletter
\newcommand\makebig[2]{%
  \@xp\newcommand\@xp*\csname#1\endcsname{\bBigg@{#2}}%
  \@xp\newcommand\@xp*\csname#1l\endcsname{\@xp\mathopen\csname#1\endcsname}%
  \@xp\newcommand\@xp*\csname#1r\endcsname{\@xp\mathclose\csname#1\endcsname}%
}
\makeatother

\makebig{biggg} {3.0}
\makebig{Biggg} {3.5}
\makebig{bigggg}{4.0}
\makebig{Bigggg}{4.5}

\usepackage{mathptmx}
\usepackage{graphicx}

\def\badquotient{\rotatebox[origin=c]{70}{$<$}}
\def\nonsingular{\rotatebox[origin=c]{70}{$\subset$}}
\def\orbifold{\rotatebox[origin=c]{70}{$\prec$}}

\newcommand{\orbishape}{\badquotient}
\newcommand{\orbiflat}{\nonsingular}
\newcommand{\orbisharp}{\orbifold}

\DeclareMathOperator{\supersharp}{\mathrm{Rh}}

\newcommand{\simplicial}{
  \Delta
}

\newcommand{\shape}{
  \mathchoice{\raisebox{1pt}{\rm\normalfont\textesh}}
             {\raisebox{1pt}{\rm\normalfont\textesh}}
             {\raisebox{0.5pt}{\rm\normalfont\scriptsize\textesh}}
             {\raisebox{0.2pt}{\rm\normalfont\tiny\textesh}}
}

\newcommand{\Localization}[1]{
  \mathrm{Loc}^{\scalebox{.6}{$#1$}}
}

\newcommand{\SimplicialLocalization}[1]{
  \Localization{#1}_{\simplicial}
}

\newcommand{\TwoLocalizationProjection}[1]{
  \SimplicialLocalization{\WeakEquivalences}
}

\newcommand{\WeakEquivalences}{
  \mathrm{WEqs}
}

\newcommand{\Sheaves}{
  \mathrm{Sh}
}

\newcommand{\InfinitySheaves}{
  \Sheaves_{\infty}
}

\makeatletter
\newif\if@sup
\newtoks\@sups
\def\append@sup#1{\edef\act{\noexpand\@sups={\the\@sups #1}}\act}%
\def\reset@sup{\@supfalse\@sups={}}%
\def\mk@scripts#1#2{\if #2/ \if@sup ^{\the\@sups}\fi \else%
  \ifx #1_ \if@sup ^{\the\@sups}\reset@sup \fi {}_{#2}%
  \else \append@sup#2 \@suptrue \fi%
  \expandafter\mk@scripts\fi}
\def\tensor#1#2{\reset@sup#1\mk@scripts#2_/}
\def\multiscripts#1#2#3{\reset@sup{}\mk@scripts#1_/#2%
  \reset@sup\mk@scripts#3_/}
\makeatother

\makeatletter
\newbox\slashbox \setbox\slashbox=\hbox{$/$}
\def\itex@pslash#1{\setbox\@tempboxa=\hbox{$#1$}
  \@tempdima=0.5\wd\slashbox \advance\@tempdima 0.5\wd\@tempboxa
  \copy\slashbox \kern-\@tempdima \box\@tempboxa}
\def\slash{\protect\itex@pslash}
\makeatother

\def\clap#1{\hbox to 0pt{\hss#1\hss}}

\def\mathclap{\mathpalette\mathclapinternal}

\def\mathclapinternal#1#2{\clap{$\mathsurround=0pt#1{#2}$}}

\renewcommand{\emph}{\textit}

\newcommand*{\Ca}{\mathcal{C}}

\newcommand*{\Ea}{\mathcal{E}}
\newcommand*{\Fa}{\mathcal{F}}

\newcommand*{\Pa}{\mathcal{P}}

\newcommand*{\Sa}{\mathcal{S}}

\newcommand*{\Nb}{\mathbb{N}}

\newcommand*{\Rb}{\mathbb{R}}
\newcommand*{\Sb}{\mathbb{S}}

\newcommand*{\Zb}{\mathbb{Z}}

\newcommand*{\im}{\mathsf{im}}

\newcommand*{\op}{^\text{op}}
\newcommand*{\inv}{^{-1}}
\newcommand*{\id}{\mathsf{id}}

\DeclareMathOperator{\colim}{\mathsf{colim}}

\newcommand*{\Prop}{\mathbf{Prop}}
\newcommand*{\Type}{\mathbf{Type}}

\DeclareMathOperator{\fib}{\mathsf{fib}}

\DeclareMathOperator{\fst}{\term{fst}}
\DeclareMathOperator{\snd}{\term{snd}}

\newcommand*{\Set}{\mathbf{Set}}

\newcommand*{\Ho}{\mathcal{S}} 

\newcommand{\ShOn}{
  \InfinitySheaves
}

\newcommand{\dprod}[1]{({#1}) \to }           
\newcommand{\dsum}[1]{({#1}) \times }       
\newcommand{\type}[1]{\mathsf{{#1}}}
\newcommand{\term}[1]{\mathsf{{#1}}}

\newcommand{\trunc}[1]{\left\lVert#1\right\rVert}       
\newcommand{\B}{\mathbf{B}}
\newcommand{\pt}{\term{pt}}
\newcommand{\ord}[1]{\underline{\type{#1}}}

\DeclareMathOperator{\csk}{\type{csk}_0}
\DeclareMathOperator{\sk}{\type{sk}_0}

\DeclareMathOperator{\geomreal}{\type{re}}

\newcommand{\xto}[1]{\xrightarrow{#1}}

\newcommand{\pto}{\,\cdot\kern-.1em{\to}\,}

\makeatletter
\providecommand*{\xmapstofill@}{%
  \arrowfill@{\mapstochar\relbar}\relbar\rightarrow
}
\providecommand*{\xmapsto}[2][]{%
  \ext@arrow 0395\xmapstofill@{#1}{#2}%
}

\makeatletter
\def\slashedarrowfill@#1#2#3#4#5{%
  $\m@th\thickmuskip0mu\medmuskip\thickmuskip\thinmuskip\thickmuskip
   \relax#5#1\mkern-7mu%
   \cleaders\hbox{$#5\mkern-2mu#2\mkern-2mu$}\hfill
   \mathclap{#3}\mathclap{#2}%
   \cleaders\hbox{$#5\mkern-2mu#2\mkern-2mu$}\hfill
   \mkern-7mu#4$%
}
\def\rightslashedarrowfill@{%
  \slashedarrowfill@\relbar\relbar\mapstochar\rightarrow}
\newcommand\xslashedrightarrow[2][]{%
  \ext@arrow 0055{\rightslashedarrowfill@}{#1}{#2}}
\makeatother

\tikzset{
    vert/.style={anchor=south, rotate=90, inner sep=.5mm}
} 

\usepackage{mathpartir}

\newcommand{\rulen}[1]{\textsc{#1}}
\newcommand{\yields}{\vdash}

\newcommand{\judge}{\!\!\!\mathcal{J}}

\newcommand{\ctx}{\,\,\mathsf{ctx}}
\newcommand{\jtype}{\,\,\mathsf{type}}

\newcommand{\ind}[3]{\mathsf{let} \; {#2} := {#1} \, \mathsf{in} \, {#3}}

\let\oldequiv\equiv%
\renewcommand{\equiv}{\simeq}
\newcommand{\defeq}{\oldequiv}

\DeclareSymbolFont{extraup}{U}{zavm}{m}{n}
\DeclareMathSymbol{\arevfilledspade}{\mathalpha}{extraup}{81}
\DeclareMathSymbol{\arevopenheart}{\mathalpha}{extraup}{82}
\DeclareMathSymbol{\arevopendiamond}{\mathalpha}{extraup}{83}
\DeclareMathSymbol{\arevfilledclub}{\mathalpha}{extraup}{84}
\DeclareMathSymbol{\arevopenspade}{\mathalpha}{extraup}{85}
\DeclareMathSymbol{\arevfilledheart}{\mathalpha}{extraup}{86}
\DeclareMathSymbol{\arevfilleddiamond}{\mathalpha}{extraup}{87}
\DeclareMathSymbol{\arevopenclub}{\mathalpha}{extraup}{88}

\definecolor{darkeryellow}{rgb}{1.0, 0.77, 0.05}

\newcommand{\focusAcolor}[1]{\textcolor{red!50!white}{#1}}
\newcommand{\focusBcolor}[1]{\textcolor{blue!60!white}{#1}}
\newcommand{\focusCcolor}[1]{\textcolor{darkeryellow}{#1}}

\newcommand{\focusA}{\focusAcolor{\arevfilledheart}}
\newcommand{\focusB}{\focusBcolor{\arevfilledclub}}
\newcommand{\focusC}{\focusCcolor{\arevfilledspade}}

\newcommand{\Afam}{\textcolor{red!65!white}{G}}
\newcommand{\Bfam}{\textcolor{blue!75!white}{H}}

\newcommand{\lock}{\text{\faLock}}

\usepackage{biblatex}
\newtheorem{thm}{Theorem}[subsection]

\theoremstyle{definition}
\newtheorem{defn}[thm]{Definition}

\newtheorem{axiom}{Axiom}

\newtheorem{rmk}[thm]{Remark}

\newtheorem*{acknowledgements}{Acknowledgements}

\newtheorem{lem}[thm]{Lemma}

\newtheorem{cor}[thm]{Corollary}
\newtheorem{prop}[thm]{Proposition}

\title{Commuting Cohesions}
\author{David Jaz Myers \and Mitchell Riley}

\begin{document}

\maketitle

\begin{abstract}
  Shulman's spatial type theory internalizes the modalities of Lawvere's axiomatic cohesion in a homotopy type theory, enabling many of the constructions from Schreiber's modal approach to differential cohomology to be carried out synthetically. In spatial type theory, every type carries a spatial cohesion among its points and every function is continuous with respect to this. But in mathematical practice, objects may be spatial in more than one way at the same time; a simplicial space has both topological and simplicial structures. Moreover, many of the constructions of Schreiber's differential cohomology and Schreiber and Sati's account of proper equivariant orbifold cohomology require the interplay of multiple sorts of spatiality --- differential, equivariant, and simplicial.

  In this paper, we put forward a type theory with ``commuting focuses'' which allows for types to carry multiple kinds of spatial structure. The theory is a relatively painless extension of spatial type theory, and enables us to give a synthetic account of simplicial, differential, equivariant, and other cohesions carried by the same types. We demonstrate the theory by showing that the homotopy type of any differential stack may be computed from a discrete simplicial set derived from the {\v C}ech nerve of any good cover. We also give other examples of multiple cohesions, such as differential equivariant types and supergeometric types, laying the groundwork for a synthetic account of Schreiber and Sati's proper orbifold cohomology.
\end{abstract}

\tableofcontents

\section{Introduction}

Homotopy type theory is a novel foundation of mathematics which centers the notion of identification of mathematical objects. In homotopy type theory, every mathematical object is of a certain \emph{type} of mathematical object; and, if $x$ and $y$ are both objects of type $X$, then we know by virtue of the definition of the type $X$ what it means to identify $x$ with $y$ as elements of the type $X$. For example, if $x$ and $y$ were real vector spaces (so that $X$ was the type of real vector spaces), then to identify $x$ with $y$ would be to give a $\Rb$-linear isomorphism between them. If $x$ and $y$ were smooth manifolds, then to identify them would be to give a diffeomorphism between them. If $x$ and $y$ were mere numbers, then to identify them would be simply to prove them equal. And so on, for any type of mathematical object.

Homotopy theory, in the abstract, is the study of the identifications of mathematical objects. Homotopy type theory is well suited for synthetic homotopy theory (e.g. \cite{brunerie:thesis, hfll:blakers-massey, bvr:higher-groups-in-hott, cors:localization-in-hott} and many others), but to apply these theorems in algebraic topology --- where objects are identified by giving continuous deformations of one into the other --- requires a modification to the theory. To emphasize the difference here, compare the higher inductive circle $S^{1}$, which is the type freely generated by a point with a self-identification, with the topological circle $\Sb^{1}$ defined as the set of points in the real plane with unit distance from the origin:
\[
\Sb^{{1}} \defeq \{(x, y) : \Rb^{2} \mid x^{2} + y^{2} = 1\}.
\]
The base point of the higher inductive circle $S^{1}$ has many non-trivial self-identifications, whereas two points of the topological circle may be identified (in a unique way) just when they are equal. The two types are closely related however: the higher inductive circle $S^{1}$ is the \emph{homotopy type} of the topological circle $\Sb^{1}$ obtained by identifying the points of the latter by continuous deformation. Ordinary homotopy type theory does not have the language to express this relationship, and therefore cannot apply the synthetic theorems concerning the higher inductive circle to topological questions about the topological circle.

What is needed is a way to distinguish between types which carry topological structure and discrete types with only homotopical structure. In his \emph{Cantor's `Lauter Einsen' and Cohesive Toposes} \cite{lawvere:cantor}, Lawvere points out that this distinction between natively cohesive and discrete sets is already present in the writings of Cantor as the distinction between the \emph{Menge} of mathematical practice and the abstract \emph{Kardinalzahlen} which arise by abstracting away from the relationships among the points of a space. In the paper, and his subsequent \emph{Axiomatic Cohesion} \cite{lawvere:cohesion}, Lawvere formalizes this opposition between cohesion and discreteness as an adjoint triple between toposes:
\[
\begin{tikzcd}
\mbox{Mengen} \ar[d, "points" description] \ar[d, hookleftarrow, shift left = 5,
bend left, "codiscrete"] \ar[d, hookleftarrow, shift right = 5, bend right,
"discrete"']\\
\mbox{Kardinalen}
\end{tikzcd}
\]

This adjoint triple induces an adjoint pair of idempotent (co)monads on the topos of spaces or \emph{Mengen}: the left adjoint, $\flat$, retopologizes a space with the discrete topology, and the right adjoint, $\sharp$, retopologizes it with the codiscrete topology. Lawvere notes that in many cases --- when the spaces in question are ``locally connected'' --- there will be a fourth adjoint $\pi_{0}$ on the left which produces the discrete set of connected components of a space; this system of adjoint functors characterizes his axiomatic cohesion.

But the real power of Lawvere's axiomatic cohesion is unlocked by Schreiber's move from
$1$-toposes whose objects are cohesive sets to $\infty$-toposes whose objects
are cohesive \emph{homotopy types}. In his \emph{Differential Cohomology in
  a Cohesive $\infty$-Topos} (DCCT) \cite{urs:diff-coh},
Schreiber shows that Lawvere's axiomatics, when interpreted in $\infty$-toposes,
give rise to the hexagonal fracture diagrams which characterize differential
cohomology --- alongside many other observations about the centrality of the
defining adjoints of cohesion in higher topology and physics. What was the
functor $\pi_0$ that took the connected components of a space becomes, in the
$\infty$-categorical setting, the functor $\Pi_{\infty}$ which takes the
\emph{shape} (in the sense of Lurie \cite[\S 7.1.6]{lurie:htt}) of a stack. All
in all, a cohesive $\infty$-logos has three adjoint endofunctors
\[
\shape \dashv \flat \dashv \sharp
\]
where $\shape$ takes the \emph{shape} or homotopy type of a higher space
considered as a discrete space,
$\flat$ takes its underlying homotopy type of discrete points, and $\sharp$
takes the underlying homotopy type of points but retopologized codiscretely.

In \emph{Brouwer's Fixed Point Theorem in Real-Cohesive Homotopy Type Theory}
\cite{mike:real-cohesive-hott} (henceforth \emph{Real Cohesion}), Shulman brings this distinction between cohesive \emph{Mengen} and discrete \emph{Karndinalen} to homotopy type theory via his \emph{spatial type theory}. Spatial type theory internalizes the $\flat$ and $\sharp$ modalities from Schreiber's DCCT which relate discrete (but homotopically interesting) types like $S^{1}$ and spatial types like $\Sb^{1}$. Spatial type theory also improves upon a previous axiomatization of these modalities in HoTT due to Schreiber and Shulman \cite{schreiber-shulman:cohesive}, by replacing axioms with judgemental rules. \emph{Cohesive homotopy type theory} is spatial type theory with an additional axiom that implies the local contractibility of the sorts of spaces in question; from this axiom the further left adjoint $\shape$ to $\flat$ may be defined.

Homotopy type theory may be interpreted into any $\infty$-topos \cite{kapulkin-lumsdaine:simplicial,mike:all}, so that a type in homotopy type theory becomes a sheaf of homotopy types externally. In particular, if we interpret the topological circle $\Sb^{1}$ defined as a subset of $\Rb^{2}$ into the $\infty$-topos of sheaves on the site of continuous manifolds, it becomes the sheaf (of sets) represented by the external continuous manifold $\Sb^{1}$, while the higher inductive circle $S^{1}$ gets interpreted as the constant sheaf at the homotopy type of the circle. By the Yoneda lemma, then, any function definable on $\Sb^{1}$ is necessarily continuous. Since functions $f : X \to Y$ in HoTT are defined simply by specifying an element $f(x) : Y$ in the context of a free variable $x : X$, variation in a free variable confers a liminal sort of continuity: such an expression could be interpreted in a spatial $\infty$-topos in which case it necessarily defines a continuous function.

Shulman's spatial type theory works by introducing the notion of a \emph{crisp free variable} to get around this liminal continuity. An expression in spatial type theory depends on its crisp free variables \emph{discontinuously}. The modalities $\flat$ and $\sharp$ of spatial type theory represent crisp variables universally on the left and right respectively. In this way, $\flat X$ is the discrete retopologization of the spatial type $X$, while $\sharp X$ is its codiscrete retopologization --- a map out of $\flat X$ is a discontinuous map out of $X$, while a map into $\sharp X$ is a discontinuous map into $X$.

Spatial type theory is intended to be interpreted into \emph{local} geometric morphisms $\gamma : \Ea \to \Sa$ of $\infty$-toposes, those for which $\gamma_{\ast}$ has a fully faithful right adjoint $\gamma^{!}$ which gives a geometric morphism $f : \Sa \to \Ea$ (with $f_{\ast} := \gamma^{!}$) adjoint to $\gamma$ which acts as the \emph{focal point} of $\Ea$ as a space over $\Sa$. The adjoint modalities $\flat$ and $\sharp$ are interpreted as the adjoint idempotent (co)monad pair $\gamma^{\ast}\gamma_{\ast}$ and $\gamma^{!}\gamma_{\ast}$ respectively. A crisp free variable is then one which varies over an object of the focal point $\Sa$: a free variable is crisp when it is \emph{in focus}.

There is not only one way for mathematical objects to be spatial. Spaces may cohere with smooth, analytic, algebraic, condensed, and simplicial or cubical combinatorial structures --- and more. Each of these cases would give rise to a particular spatial type theory as the internal language of an appropriate local $\infty$-topos. But there are many cases arising in practice where we need not just one axis of spatiality, but many at once. For example, it is a classical theorem that the homotopy type of a manifold may be computed as the realization of a (topologically discrete) simplicial set associated to the {\v C}ech nerve of a good open cover of the manifold. This theorem relates a simplicial set to a continuous space, via an intermediary simplicial space which is both continuous and simplicial at the same time --- the {\v C}ech nerve of the cover. But in spatial homotopy type theory there is only one notion of crisp variable, and therefore just one sort of spatiality.

For simplicial types, the discrete reflection is the $0$-skeleton $\sk$, while the codiscrete reflection is the $0$-coskeleton. For simplicial spaces, we then have both the (topologically) discrete $\flat$ and codiscrete $\sharp$, as well as the simplicially $0$-skeletal $\sk$ and $0$-coskeletal $\csk$. Interestingly, the {\v C}ech nerve itself arises from these modalities: the {\v C}ech nerve of a map $f : X\to Y$ between $0$-skeletal types (that is, continuous or differential stacks with no simplicial structure) is its $\csk$-image, as we will see later in \cref{prop:cech.nerve}. A simplicial space has both a shape $\shape X$ and a realization (or colimit) $\geomreal X$; the first is a topologically discrete simplicial type, while the latter is a $0$-skeletal but spatial type. With all these modalities, we can prove the theorem about good covers described above as \cref{thm:good.cover.homotopy}.

Another use case for multiple axes of spatiality is Sati and Schreiber's \emph{Proper orbifold cohomology} \cite{sati-schreiber:proper-orbifold-cohomology}, where orbifolds are understood both as having both differential structure (as differential stacks) and global equivariant structure (concerning their singularities). In order to get the correct generalized cohomology of orbifolds without relying on ad-hoc constructions based on a global quotient presentation of the orbifold, Sati and Schreiber work with the $\infty$-topos of global equivariant differential stacks, which is local both over the global equivariant topos and the topos of differential stacks. Here the differential modalities $\shape$, $\flat$ and $\sharp$ are augmented with the modalities of equivariant cohesion \cite{rezk:global-cohesion}: $\orbishape$, $\orbiflat$, and $\orbisharp$, which take the strict quotient, the underlying space as an invariant type, and the Yoneda embedding of the underlying space of a global equivariant type respectively. Again the modalities play a central role in the theory, with the ordinary Borel cohomology of a global quotient orbifold $X \sslash G$ being the ordinary cohomology of $\shape \orbiflat (X \sslash G)$, while the proper equivariant Bredon cohomology of $X \sslash G$ is the cohomology of $\orbisharp(X \sslash G)$, twisted by the map to $\orbisharp \B G$ classifying the quotient map $\orbisharp X \to \orbisharp(X \sslash G)$.

In these cases, modalities that lie in the same position in their adjoint chain commute with each other, so, for example, $\flat$ commutes with $\sk$ and $\sharp$ commutes with $\csk$. However, there are cases where these modalities are nested, with one spatiality being a refinement of another. This occurs for example in supergeometry as formulated by Schreiber in \cite{urs:diff-coh} with the modalities of \emph{solid cohesion}. The supergeometric focus is given by the even comodality $\rightrightarrows$ (which takes the even part of a superspace) and the rheonomic modality $\supersharp$ which is given by localizing at the \emph{odd line} $\Rb^{0\mid 1}$.

In this paper, we put forward a modification of spatial type theory to allow for
multiple axes of spatiality. Our theory works by allowing for a meet
semi-lattice of \emph{focuses} $\focusA, \focusB,\dots$, each with a separate
notion of $\focusA$-crisp variable and pair of adjoint (co)modalities
$\flat_{\focusA}$ and $\sharp_{\focusA}$. Like spatial type theory, our custom
type theory gets us to the coalface of synthetic homotopy theory very
efficiently while staying simple enough to be used in an informal style.

The presence of multiple notions of crispness forces a more complex context
structure than spatial type theory's separation of the context into a crisp zone
and cohesive zone. Similar to many other modal type theories \cite{lsr:multi,
  mtt, mtt-fitch, atkey:qtt, nvd:parametric-quantifiers}, we annotate each
variable with modal information, here, the focuses for which that variable is
crisp. The typing rules for the modalities of each focus then work essentially
independently. The exception is $\flat$-elimination, which is upgraded to allow
the crispness of the term being eliminated to be maintained in the variable
bound by the induction (a `crisp' induction principle).

Ours is far from the only extension of type theory with multiple modalities, but
as we discuss in more detail later, no existing theory has the combination of
features that we are looking for: dependent types (ruling out \cite{lsr:multi})
that may depend on modal variables (ruling out \cite{atkey:qtt}), multiple
commuting comodalities (ruling out \cite{mike:real-cohesive-hott, drats,
  nvd:parametric-quantifiers}) each with a with right-adjoint modality (ruling
out \cite{grtt}) and no further left-adjoints (ruling out \cite{mtt, mtt-fitch}
and \cite[\S 14]{cavallo:thesis}).

In addition to allowing us to formalize the theorem about {\v C}ech nerves of open covers as \cref{thm:good.cover.homotopy}, our type theory will be able to handle the equivariant differential cohesion used by Sati and Schreiber in their \emph{Proper orbifold cohomology} \cite{sati-schreiber:proper-orbifold-cohomology}, as well as the nested focuses of Schreiber's supergeometric \emph{solid} cohesion \cite{urs:diff-coh}. This extends the work
of Cherubini \cite{Cherubini:Thesis} and the first author
\cite{jaz:good-fibrations,jaz:modal-fracture,jaz:orbifolds} of
giving synthetic accounts of the constructions of Schreiber
\cite{urs:diff-coh} and Sati-Schreiber
\cite{sati-schreiber:proper-orbifold-cohomology}.

Positing an additional focus does not disturb arguments made using existing focuses, so we also expect our theory to be helpful when dipping into simplicial arguments in the course of other reasoning by adding a simplicial focus and making use of the new modalities. The problem of defining simplicial types in ordinary Book HoTT remains open, and there are now a number of different approaches to constructing simplicial types which each use some extension to the underlying type theory. In this paper, we will axiomatize the $1$-simplex $\Delta[1]$ as a linear order with distinct top and bottom and use the cohesive modalities to define the {\v C}ech nerve of a map and the realization or colimit of a simplicial type. We believe our approach here would pair nicely with other approaches to simplicial types for the purposes of synthetic $(\infty,1)$-category theory such as \cite{rs:synthetic-cats,buccholtz-weinberger:fibered.cats,weinberger:cartesian.fibrations,weinberger:internal.sums}, where the $\sk$ modality would take the core of a Rezk type.\footnote{Though we have not looked in detail at how the focuses would work with the Riehl-Shulman simplicial type theory, and in particular how they would interact with the cubes/topes zones of the Riehl-Shulman context.}

\paragraph{Outline of the present paper.} After presenting our type theory in
\cref{sec:rules}, we will look at
ways to specialize the spatiality of
a focus in \cref{sec:specializing}.
In particular, we will observe that
in many cases there is a small class
of test spaces $G_i$ so that
codiscreteness (that is, being
$\sharp$-modal) is detected by
uniquely lifting against the
$\flat$-counits $\flat G_i \to G_i$; such $G_i$ will be said to \emph{detect continuity}.
Externally, the $G_i$ could be any
family which generates the logos
under colimits. In practice, the
$G_i$ will be test spaces which
minimally carry the appropriate
spatiality: in the simplicial case,
the simplices $\Delta[n]$; in the
real-cohesive case, the Euclidean
spaces $\Rb^n$; for condensed sets,
the profinite sets, etc.

In \cref{sec:specializing}, we will
also meet a family of axioms which
hold for spatialities that are
\emph{locally contractible}. For
example, continuous manifolds which
are built from Euclidean spaces by
colimits are locally contractible,
while condensed sets which are built
from profinite sets by colimits need
not be locally contractible. In
general, a space is locally
contractible when it has a constant
\emph{shape} in the sense of Lurie
\cite[\S 7.1.6]{lurie:htt}.

We may define a space $C$ to be
contractible when any map $C \to
S$ to a discrete space $S$ is constant. If the converse
holds --- a space $S$ is discrete
($\flat$-modal) if every map $C \to
S$ is constant --- then we say that
$C$ \emph{detects the connectivity}
of spaces. For example, $\Rb$
detects the connectivity of
continuous $\infty$-groupoids, and
$\Delta[1]$ detects the connectivity
of simplicial $\infty$-groupoids.
If there is a space (or family of
spaces) which detects connectivity,
then the local geometric morphism
$p$ corresponding to the morphism is
furthermore strongly locally contractible in
that $p^{\ast}$ has a left adjoint
$p_{!}$ which takes the (constant
value of the) shape of a space. In
the case that $p$ is both local and
strongly locally contractible, we
say that $p$ is \emph{cohesive}
following Lawvere
\cite{lawvere:cohesion},
Schreiber
\cite{urs:diff-coh},
and Shulman
\cite{mike:real-cohesive-hott}.
Nullifying at the family of spaces
which detect connectivity gives a
modality $\shape$ which is left
adjoint to $\flat$; it may be
thought of as taking the homotopy
type of a space.

In \cref{sec:single.examples}
we will give example axioms for
specializing single focuses. We will
review Shulman's axioms for
\emph{real cohesion}, where the
Euclidean spaces $\Rb^n$ detect continuity
and connectivity. We will then see
simplicial cohesion in some detail,
where the simplices $\Delta[n]$
detect continuity and connectivity.
We give our types simplicial
structure by axiomatizing the
$1$-simplex $\Delta[1]$ as a
total order with distinct top and bottom
elements, following Joyal's
characterization of simplicial sets
as the classifying topos for such
orders \cite{wraith:generic-interval}. We use the $\csk$ modality
to construct {\v C}ech nerves of
maps.
Then we will describe the global
equivariant cohesion first observed
by Rezk
\cite{rezk:global-cohesion} and used by Sati and Schreiber in \cite{sati-schreiber:proper-orbifold-cohomology}.
Finally, we will briefly describe
axioms for topological focuses such
as Johnstone's topological topos of
sequential spaces \cite{johnstone:topological-topos} and the
condensed/pyknotic topos of
Clausen-Scholze \cite{clausen-sholze:condensed}
and Barwick-Haine \cite{barwick-haine:pyknotic}.

After surveying some of the
different sorts of spatiality which
types might carry, we turn our
attention to multiple focuses in
\cref{sec:multiple.focuses}. In
\cref{defn:orthogonal.cohesions}, we
define what it means for two
cohesions to be \emph{orthogonal}:
when the family which detects the
connectivity of one is discrete with
respect to the other, and vice
versa. We then prove a few lemmas
concerning orthogonal cohesions, in particular concerning when it is possible to
commute the various modalities past each other.

Finally, we give examples of
multiple focuses
in \cref{sec:multiple.examples}. We
begin with simplicial real cohesion,
which has both a simplicial focus
and a real-cohesive focus which are
orthogonal. We prove, in
\cref{thm:good.cover.homotopy}, that
the shape of any $0$-skeletal type
$M$ may be computed as the
realization of a topologically discrete simplicial type
constructed from the {\v C}ech nerve
of any \emph{good} cover $U$ of $M$ ---
one for which finite intersections
of the $U_i$ are contractible in the
sense of being $\shape$-connected.

Next, we combine equivariant cohesion with differential cohesion to give the
series of modalities used in Sati and Schreiber's \emph{Proper orbifold
  cohomology} \cite{sati-schreiber:proper-orbifold-cohomology}. Happily, no extra
axioms are needed to show that the two cohesions are orthogonal; we prove this
in \cref{lem:equivariant.differential.orthogonal}.

Finally, we describe the supergeometric
or ``solid'' cohesion of Schreiber's
\emph{Differential Cohomology in a
  Cohesive $\infty$-topos}. This
extends real cohesion with the
\emph{odd line} $\Rb^{0\mid 1}$,
where the ``discrete'' comodality of
the supergeometric focus takes the
even part of a supergeometric space,
and the ``codiscrete'' modality
takes a \emph{rheonomic} reflection
of the space, one whose super
structure is uniquely determined by
its even structure. Unlike our other
examples where the focuses involved
are orthogonal, here the
differential focus is included in
the supergeometric focus: any
discrete space is also purely even,
as is any codiscrete space.

\begin{acknowledgements}
We would like to thank Urs Schreiber for his careful reading and extensive
comments during the drafting process of the paper. And we would like to thank
Hisham Sati for his feedback and words of encouragement. The authors are
grateful for the support of Tamkeen under the NYU Abu Dhabi Research Institute grant CG008.
\end{acknowledgements}

\section{A Type Theory with Commuting Focuses}\label{sec:rules}

The fundamental duality in higher
topos theory is between the
$\infty$-topos --- a general sort of
space --- and the $\infty$-logos ---
the category of sheaves of homotopy
types on such a space~\cite{anel-joyal:topo-logie}. This duality
is perfect: a map of
$\infty$-toposes $\Ea \to \Fa$ is defined to be a
lex accessible functor $\ShOn(\Fa) \to \ShOn(\Ea)$ between their corresponding
$\infty$-logoses in the opposite
direction.

This duality between toposes and
logoses gives a nice perspective on
the distinction between the
\emph{petite} toposes, which are
used as generalized spaces in
practice, and the \emph{gros}
toposes --- or rather, their dual
logoses --- which are used as
categories \emph{of} spaces, rather
than as spaces in their own right.
Quite opposite to their names, the
petite toposes are ``big'' spaces,
while the gros toposes are ``small''
spaces; it is their dual logoses
which are correctly described by the
adjectives ``petite'' and ``gros''.
Since the logos is the category of
sheaves on the topos, or
equivalently the category of \'etale
maps into the topos, the ``larger''
the topos the more constraining the
\'etale condition becomes. For that
reason, the gros toposes have
qualitatively ``smaller'' categories
of sheaves. On the other hand, the
more general the \'etale spaces may
be, the ``smaller'' the base topos
must be. In general, the ``biggest''
logoses, the logoses \emph{of}
spaces, must correspond to the
``smallest'' toposes: those toposes
which are infinitesimal patches
around a focal point. This point of view is emphasized in Chapter 4 of DCCT \cite{urs:diff-coh}.

We may therefore, as a first pass,
identify logoses of spaces as
dual to those toposes $\Ea$ which are \emph{local} over a focal
point $\Fa$. A geometric morphism $p
: \Ea \to \Fa$ is local when it
admits a left adjoint right inverse
$f : \Fa \to \Ea$ in the
$(\infty,2)$-category of toposes which we call the
\emph{focal point} of $p$. If $\Ea$
is a topological space (that is, if
its corresponding logos $\ShOn(\Ea)$
is the
category of sheaves $\ShOn(X)$ on a sober topological
space $X$), then the terminal
geometric morphism $\gamma : \Ea \to
\Sa$ is local just when $X$ has a
focal point: a point $f \in X$ whose
only open neighborhood is the whole
of $X$. In particular, the prime
spectrum of a ring $A$ is local if and only if $A$ is a local
ring; in this case, the focal point is the
unique maximal ideal.

On the
logos side, this means that the
direct image $p_{\ast}$ admits a
fully faithful right adjoint
$p^{!}$ (which is $f_{\ast}$). All together, this gives an
adjoint triple between the
corresponding logoses:
\[
  \begin{tikzcd}
    \ShOn(\Ea) \ar[d, "p_{\ast}"]
    \ar[d, shift left = 5, hookleftarrow,
    bend left, "p^{!}"] \ar[d, shift
    right = 5, hookleftarrow, bend
    right, "p^{\ast}"'] \\
    \ShOn(\Fa)
  \end{tikzcd}
\]
Thinking of the objects of
$\ShOn(\Ea)$ as generalized spaces and
the objects of $\ShOn(\Fa)$ as mere homotopy types
(sheaves on a point), we may see the
direct image $p_{\ast}$ as taking
the underlying homotopy type of
points of a space, while $p^{\ast}$
and $p^{!}$ are the discrete and
codiscrete topologizations of bare
homotopy types, respectively. This
adjoint triple gives rise to an
adjoint pair
\[
p^{\ast}p_{\ast} \dashv p^{!}p_{\ast}
\]
of a idempotent comonad
$p^{\ast}p_{\ast}$ and idempotent
monad $p^{!}p_{\ast}$ on the logos
$\ShOn(\Ea)$. Understood as operations
on spaces, these are the discrete
and codiscrete retopologizations of
a space respectively.

Examples of local toposes with focal
point $\Fa$ having category of sheaves $\ShOn(\Fa) = \infty\type{Grpd}$ the
$\infty$-category of $\infty$-groupoids
include simplicial types
$\infty\type{Grpd}^{\Delta\op}$ (where
discrete is $0$-skeletal and
codiscrete is $0$-coskeletal),
continuous and differentiable
$\infty$-groupoids\footnote{These
  are the gros toposes of $\Ca^0$
  and $\Ca^{\infty}$ manifolds, respectively.}
$\type{Sh}_{\infty}(\{\Rb^n\})$
(where discrete means all charts are
constant, and codiscrete means that
any function valued in the set of
points is a chart), condensed
$\infty$-groupoids (where discrete
means discrete, and codiscrete means
codiscrete), and global equivariant
$\infty$-groupoids
$\infty\type{Grpd}^{\type{Glo}\op}$
(where discrete means being a
constant presheaf on the global
orbit category, and codiscrete means
being a presheaf representable by an
ordinary $\infty$-groupoid).

Shulman
\cite{mike:all} has
shown that every $\infty$-logos may
be presented by a model of homotopy type theory,
allowing reasoning conducted in
homotopy type theory to be
interpreted in any $\infty$-logos. In this sense, homotopy type theory is to
$\infty$-logoses as set theory is to the $1$-logoses of Grothendieck, Lawvere,
and Tierney.
In \emph{Brouwer's Fixed Point
  Theorem in Real-Cohesive Homotopy
  Type Theory}
\cite{mike:real-cohesive-hott},
Shulman also put forward a
\emph{spatial type theory} which may (conjecturally)
be interpreted into any local
geometric morphism. Spatial type
theory is characterized by including
an adjoint pair $\flat \dashv
\sharp$ of a lex comodality $\flat$
and lex modality $\sharp$. These are to be
interpreted as $p^{\ast}p_{\ast}$
and $p^{!}p_{\ast}$ respectively.

In spatial type theory, any type has
a spatial structure. The existence
of this spatial structure is
witnessed by the two opposite ways
that we can get rid of it: either we
can remove all the spatial
relationships between points, using
the ``discrete'' $\flat$ comodality, or we can
trivialize the spatial relations
using the ``codiscrete'' $\sharp$
modality. We emphasize that this
spatial structure is distinct from
the \emph{homotopical} structure
that all types have by virtue of the
identifications between their
elements. For example, the topological circle
\[
\Sb^1 := \{(x, y) : \Rb^2 \mid x^2 +
y^2 = 1\}
\]
has a spatial structure as a subset
of the Euclidean plane (as a sheaf
on the site of continuous manifolds,
for example), but is a
homotopy $0$-type (or ``set'')
without any non-trivial
identifications between its points;
in particular $\Omega \Sb^1 = \ast$.
The homotopy type $S^1$ of the
circle, however, is spatially
discrete but has many non-trivial
identifications of its point: in
particular $\Omega S^1 = \Zb$.

There is not, however, only one way
to be spatial in mathematics. For
example a simplicial topological
space has both a simplicial
structure and a topological
structure. This can be witnessed at
the level of toposes as well. If $p
: \Ea \to \Fa$ admits a focal point
$f : \Fa \to \Ea$, then
$f^{\Delta\op} : \Fa^{\Delta\op} \to
\Ea^{\Delta\op}$ is also a focal
point of
$p^{\Delta\op} : \Ea^{\Delta\op} \to
\Fa^{\Delta\op}$, where the logos
$\ShOn(\Ea^{\Delta\op}) :=
(\ShOn(\Ea))^{\Delta\op}$ is the
category of simplicial objects in
the logos $\ShOn(\Ea)$. But there is another local
geometric morphism
$\gamma : \Ea^{\Delta\op} \to \Ea$
where $\gamma_{\ast}$
sends a simplicial sheaf
$X_{\bullet}$ to $X_0$ and
$\gamma^{!}$ is given by the
\emph{$0$-coskeleton}
$\csk S_{n} := S^n$. These two
different axes of spatiality on the
objects of $\ShOn(\Ea)^{\Delta\op}$
commute, in that the following
diagram of adjoints commutes:
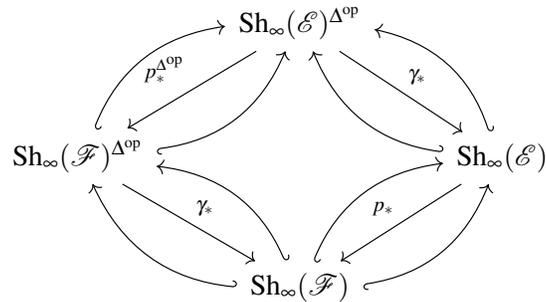
\[
  \begin{tikzcd}[row sep = large]
      & {\ShOn(\Ea)^{\Delta\op}}
      \ar[dl, hookleftarrow, bend left, end anchor={[shift={(0pt,3pt)}]}]
      \ar[dl, hookleftarrow, bend right]
      \ar[dr, hookleftarrow, bend left]
      \ar[dr, hookleftarrow, bend right, end anchor={[shift={(0pt,1pt)}]}] \\
      {\ShOn(\Fa)^{\Delta\op}}
      \ar[dr, hookleftarrow, bend left, start anchor={[shift={(0pt,-3pt)}]}]
      \ar[dr, hookleftarrow, bend right] &&
      {\ShOn(\Ea)}
      \ar[dl, hookleftarrow, bend left]
      \ar[dl, hookleftarrow, bend right, start anchor={[shift={(0pt,-1pt)}]}] \\
      & {\ShOn(\Fa)}
      \arrow["{p_\ast^{\Delta\op}}"', from=1-2, to=2-1]
      \arrow["{\gamma_{\ast}}", from=2-1, to=3-2]
      \arrow["{\gamma_{\ast}}", from=1-2, to=2-3]
      \arrow["{p_{\ast}}"', from=2-3, to=3-2]
    \end{tikzcd}
\]

In particular, we have that
$p^{\ast\Delta\op}p_{\ast}{}^{\Delta\op}$
and $\gamma^{\ast}\gamma_{\ast}$
commute as endofunctors of
$\ShOn(\Ea)^{\Delta\op}$. The former
discretely retopologizes a
simplicial space, while the latter
includes the space of $0$-simplices
as a $0$-skeletal simplicial space. Each focus gives an axis along which
the objects of the top logos
$\ShOn(\Ea)^{\Delta\op}$ may carry
spatial structure.

When working in, say, simplicial differential spaces, we would like to have
access to both the $\shape \dashv \flat \dashv \sharp$ of real cohesion and the
$\geomreal \dashv \sk \dashv \csk$ of simplicial cohesion. Shulman's spatial
type theory offers no way to do this: the
$\flat$ and $\sk$ comonads have incompatible claims on the notion of `crisp'
variable.

The solution is to allow a separate notion of crispness for each focus we are
interested in. In this section, we will describe the rules for a type theory
with commuting focuses, generalizing ordinary spatial type theory in the case of
a single non-trivial focus. We will then describe axioms which make these
into commuting \emph{cohesions}, in the sense of cohesive type theory.


To this end, we will fix an commutative idempotent monoid $\type{Focus}$ of
focuses; we will write the product of the focus $\focusA$ and the focus
$\focusB$ as $\focusA\focusB$. This product induces an ordering on the focuses
by saying that $\focusB \leq \focusA$ whenever $\focusB\focusA = \focusB$. With
respect to this ordering, the product becomes the meet; we may therefore also
think of $\type{Focus}$ as a meet semi-lattice. We will write the identity
element of $\type{Focus}$ as $\top$, and note that it is the top focus with
respect to the order.

For most of our purposes in this paper, our commutative idempotent monoid
$\type{Focus}$ of focuses will be freely generated by a finite set of basic
focuses.  Explicitly, we may take
$\type{Focus} = \Pa_{f}(\type{BasicFocuses})\op$ to be the set of finite subsets
of the set of basic focuses with union as the product, and therefore the
opposite of the ordering of subsets by inclusion.

All variables in the context will be annotated with the focus that they are in:
\[
x :_{\focusA} X \vdash t : T
\]
In general, we will abbreviate the context entry $x :_{\top} X$ as $x : X$.  In
the case that $\type{Focus} = \{\focusA \leq \top\}$ is freely generated by one
basic focus, we recover the split context used in Shulman's spatial type theory,
where our context $x :_{\focusA} X, \, y :_{\top} Y \ctx$ corresponds to
Shulman's context $x :: X \mid y : Y \ctx$.

To describe the typing rules, we will need a couple of auxiliary operations on
contexts. The first operation $\focusA\Gamma$ adds a specific focus
$\focusA$ to the annotation on every
variable in a context. So:
\begin{align*}
  \focusA(\cdot) &:\defeq \cdot \\
  \focusA(\Gamma, x :_{\focusB} A) &:\defeq (\focusA \Gamma), x :_{\focusA\focusB} A
\end{align*}
We also need an operation $\focusA \setminus \Gamma$ that deletes any variables \emph{not} contained within a
given focus $\focusA$; this is the equivalent of going from
$\Delta \mid \Gamma \ctx$ to $\Delta \mid \cdot \ctx$ in ordinary spatial type theory.
\begin{align*}
  \focusA \setminus (\cdot) &:\defeq \cdot \\
  \focusA \setminus (\Gamma, x :_{\focusB} A) &:\defeq
                               \begin{cases}
                                 (\focusA \setminus \Gamma), x :_{\focusB} A & \text{if $\focusB \leq \focusA$} \\
                                 \focusA \setminus \Gamma & \text{otherwise}
                               \end{cases}
\end{align*}

We say that a variable $x :_{\focusB} X$ is \emph{$\focusA$-crisp} if $\focusB \leq
\focusA$, and so the $\focusA$-crisp variables are precisely those that survive the $\focusA \setminus \Gamma$ operation. We say that a term $t : T$ is \emph{$\focusA$-crisp} if both it and its
type $T$ only contain
$\focusA$-crisp variables, i.e., it is well-formed in context $\focusA \setminus \Gamma$.

We are now ready to describe the rules of the type theory. All the usual type
formers --- $\Sigma$s, $\Pi$s, etc. --- will be included as usual, only
referring to variables of the top focus $\top$. By the convention that $x
:_{\top} X$ be written as $x : X$, these rules look exactly as they do
usually. We therefore focus on the new features of type theory with commuting
focuses.

\paragraph{Structural Rules.}

\begin{mathpar}
  \inferrule*[left=ctx-empty]{~}{\cdot \ctx} \and
  \inferrule*[left=ctx-ext]{\focusA \setminus \Gamma \yields A \jtype}{\Gamma, x :_{\focusA} A \ctx} \and
  \inferrule*[left=var]{\Gamma, x :_{\focusA} A, \Gamma' \ctx}{\Gamma, x :_{\focusA} A, \Gamma' \yields x : A}
\end{mathpar}

In prose, these rules read as follows:
\begin{itemize}
\item \rulen{ctx-empty}: The empty context is a context.
\item \rulen{ctx-ext}: If $A$ is a $\focusA$-crisp type in context $\Gamma$, then
  $\Gamma, x :_{\focusA} A \ctx$ is a context.
\item \rulen{var}: If $x :_{\focusA} A$ appears in a context, then the variable $x$ has type $A$ in
  that context.
\end{itemize}

\begin{rmk}
  Given a context $\Gamma, x :_{\focusA} A \ctx$, it must be the case that $A$
  only depends on the variables in $\Gamma$ which are themselves
  $\focusA$-crisp. This careful context formation rule is what replaces the
  division of the context into two zones in Shulman's spatial type theory. In
  the conclusion of the variable rule, the type $A$ is well-formed in context
  $\Gamma, x :_{\focusA} A, \Gamma' \ctx$ by the admissible \rulen{divide-wk}
  rule given below, followed by further weakening with $\Gamma'$.
\end{rmk}

\begin{rmk}
  Rather than annotating variables, may be tempting to try a floating context
  separator $\mid_{\focusA}$ for each focus, so that the variables to the left
  of $\mid_{\focusA}$ are precisely the $\focusA$-crisp ones. Such contexts are
  not sufficiently general; specifically, the $\flat$-elimination rule will let
  us produce a context containing $x :_{\focusA} A, y :_{\focusB} B$ which
  clearly cannot be separated in this way.
\end{rmk}

The following rules and equations will be made admissible, with the proofs
sketched in \cref{sec:proofs-admiss}.
\begin{mathpar}
  \inferrule*[left=wk, fraction={-{\,-\,}-}]{\Gamma, \Gamma' \yields \judge}{\Gamma,
    x :_{\focusA} A, \Gamma' \yields \judge} \and

  \inferrule*[left=subst, fraction={-{\,-\,}-}]{\focusA \setminus \Gamma
    \yields a : A \and \Gamma, x:_{\focusA} A, \Gamma' \yields \judge}{\Gamma,
    \Gamma'[a/x] \yields \judge[a/x]} \\

  \inferrule*[left=promote-ctx, fraction={-{\,-\,}-}]{\Gamma
    \ctx}{\focusA\Gamma \ctx} \and

  \inferrule*[left=promote, fraction={-{\,-\,}-}]{\Gamma
    \yields \judge}{\focusA\Gamma \yields \judge} \\

  \inferrule*[left=divide-ctx, fraction={-{\,-\,}-}]{\Gamma \ctx}{\focusA
    \setminus \Gamma \ctx} \and

  \inferrule*[left=divide-wk, fraction={-{\,-\,}-}]{\focusA \setminus \Gamma
    \yields \judge}{\Gamma \yields \judge} \\

  \focusA(\focusB \Gamma) \defeq (\focusA \focusB) \Gamma \and
  \focusA \setminus (\focusB \setminus \Gamma) \defeq (\focusB \focusA) \setminus \Gamma \and
\end{mathpar}

\begin{itemize}
\item First, we have ordinary weakening by a variable, and a `crisp'
  substitution similar to that used in spatial type theory, where crisp
  variables may only be substituted with similarly crisp terms. These specialize
  to the ordinary weakening and substitution rules when used for
  $\focusA \defeq \top$.
\item \rulen{promote-ctx} corresponds to the application of the
  endofunctor $\focusA$ to the context $\Gamma$, and \rulen{promote} to
  precomposition with the counit morphism $\focusA \Gamma \to
  \Gamma$.
\item \rulen{divide-ctx} gives the largest `subcontext'
  $\focusA \setminus \Gamma$ of $\Gamma$ such that there is a substitution
  $\Gamma \to \focusA (\focusA \setminus \Gamma)$. The context operation
  $\focusA \setminus -$ thus acts like a left-adjoint to $\focusA -$, although
  semantically a left-adjoint may not exist.
\end{itemize}

\paragraph{Rules for $\flat$.}

We now come to the rules for the $\flat$ comodality.
\begin{mathpar}
  \inferrule*[left=$\flat$-form]{\focusA \setminus \Gamma \yields A \jtype}{\Gamma \yields
    {\flat_{\focusA}} A \jtype} \\
  \inferrule*[left=$\flat$-intro]{\focusA \setminus \Gamma \yields
    M:A}{\Gamma \yields M^{\flat_{\focusA}}: {\flat_{\focusA}} A} \and
  \inferrule*[left=$\flat$-elim]{
    \focusB\focusA \setminus \Gamma \yields A \jtype \and
    \Gamma, x :_{\focusB} {\flat_{\focusA}} A \yields C \jtype \\\\
    \focusB \setminus \Gamma \yields M : {\flat_{\focusA}} A \\ \Gamma, u
    :_{\focusB\focusA} A \yields N : C[u^{\flat_{\focusA}}/x]}{\Gamma \yields
    (\ind{M}{u^{\flat_{\focusA}}}{N}):C[M/x]} \and
  \inferrule*[left=$\flat$-beta]{
    \focusB\focusA \setminus \Gamma \yields A \jtype \and
    \Gamma, x :_{\focusB} {\flat_{\focusA}} A \yields C \jtype \\\\
    \focusB\focusA \setminus \Gamma \yields K : A \\
    \Gamma, u :_{\focusB\focusA} A \yields N : C[u^{\flat_{\focusA}}/x]} {\Gamma
    \yields(\ind{K^{\flat_{\focusA}}}{u^{\flat_{\focusA}}}{N}) \defeq N[K/u] \;:\:
    C[K^{\flat_{\focusA}}/x]}
\end{mathpar}

In prose, these rules read as follows:
\begin{itemize}
\item \rulen{$\flat$-form}: If $A$ is a $\focusA$-crisp type, then we may form $\flat_{\focusA} A \jtype$.
\item \rulen{$\flat$-intro}: If $M$ is a $\focusA$-crisp term of type $A$, then we may form
  $M^{\flat_{\focusA}}$ of type $\flat_{\focusA} A$.
\item \rulen{$\flat$-elim}: If $C$ is a type depending on the $\focusB$-crisp variable
  $x :_{\focusB} \flat_{\focusA} A$, and $M : \flat_{\focusA} A$ is a
  $\focusB$-crisp element of type $\flat_{\focusA} A$, then we may assume that
  $M$ is of the form $u^{\flat_{\focusA}}$ for a $\focusB\focusA$-crisp variable
  $u :_{\focusB\focusA} A$ when defining an element of $C[M/x]$. We write this
  element as $(\ind{M}{u^{\flat_{\focusA}}}{N}) : C[M/x]$ where
  $N : C[u^{\flat_{\focusA}}/x]$ is the element we defined assuming that $M$ was
  of the form $u^{\flat_{\focusA}}$. The equation
  $\focusA \setminus (\focusB \setminus \Gamma) \defeq (\focusB \focusA)
  \setminus \Gamma$ is necessary here to know that the type $\flat_{\focusA} A$ is
  well-formed in context $\focusB \setminus \Gamma$.

\item \rulen{$\flat$-beta}: If $M$ actually is of the form $K^{\flat_{\focusA}}$ for
  suitably crisp $K$, then we simply substitute $K$ in for $u$. The
  term $K$ must be $\focusB\focusA$-crisp for both the
  \rulen{$\flat$-intro} and \rulen{$\flat$-elim} to have been applied,
  and so its substitution for the $\focusB\focusA$-crisp variable $u$
  is well-formed.
\end{itemize}

\begin{rmk}
These rules are stronger than the ones used by Shulman for spatial
type theory, even in the case of a single focus. We have built in a
\emph{$\focusB$-crisp induction principle} for $\flat_{\focusA}$, for
any two focuses $\focusA$ and $\focusB$: if the term we are inducting
on is already $\focusB$-crisp, then we may maintain that crispness in
the new assumption $u$.

If we have a single non-trivial focus $\focusA$, as is the case in
Shulman's type theory, then taking $\focusB = \focusA$ in the above
expression yields the `crisp $\flat$ induction' principle of
\cite[Lemma 5.1]{mike:real-cohesive-hott}. This induction principle is
proven by taking a detour through $\sharp$, but here we choose to
build it into the rule directly.


Our elimination rule is in fact also admissible from the less general
one that requires the freshly bound variable to only be
$\focusA$-crisp, but we choose the more general rule for
convenience.
\end{rmk}

\paragraph{Rules for $\sharp$.}
The rules for $\sharp$ are a little simpler, and in the case of a single focus
specialize exactly to the rules of spatial type theory.

\begin{mathpar}
  \inferrule*[left=$\sharp$-form]{\focusA\Gamma \yields A\jtype}{\Gamma \yields {\sharp_{\focusA}} A \jtype} \and
  \inferrule*[left=$\sharp$-intro]{\focusA\Gamma \yields M:A}{\Gamma \yields M^{\sharp_{\focusA}} : {\sharp_{\focusA}} A}\and
  \inferrule*[left=$\sharp$-elim]{\focusA \setminus \Gamma \yields N:{\sharp_{\focusA}} A}{\Gamma \yields N_{\sharp_{\focusA}} : A}\\
  \inferrule*[left=$\sharp$-beta]{\focusA \setminus \Gamma \yields M:A}{\Gamma \yields (M^{\sharp_{\focusA}})_{\sharp_{\focusA}} \defeq M : A}\and
  \inferrule*[left=$\sharp$-eta]{\Gamma \yields N:{\sharp_{\focusA}} A}{\Gamma \yields N \defeq (N_{\sharp_{\focusA}})^{\sharp_{\focusA}} :
    {\sharp_{\focusA}} A}
\end{mathpar}

In prose, these rules read as follows:
\begin{itemize}
\item \rulen{$\sharp$-form}: When forming the type $\sharp_{\focusA} A$, all variables may be
  used in $A$ as though they are $\focusA$-crisp.
\item \rulen{$\sharp$-intro}: When forming a term $M^{\sharp_{\focusA}} : \sharp_{\focusA} A$,
  all variables may be used in $M$ as though they are $\focusA$-crisp.
\item \rulen{$\sharp$-elim}: If $N$ is a $\focusA$-crisp element of $\sharp_{\focusA} A$, we
  may extract an element $N_{\sharp_{\focusA}} : A$.
\item \rulen{$\sharp$-beta}: If $M$ is a $\focusA$-crisp element of $A$, then
  $M^{\sharp_{\focusA}}{}_{\sharp_{\focusA}} \defeq M$.
\item \rulen{$\sharp$-eta}: Any term of $N : \sharp_{\focusA} A$ is definitionally equal to
  $N_{\sharp_{\focusA}}{}^{\sharp_{\focusA}}$. As in ordinary spatial type
  theory, the term $N_{\sharp_{\focusA}}$ may not be well-typed on its own,
  because it may use non-crisp variables of the context $\Gamma$. It is however
  well-typed underneath the outer $(-)^{\sharp_{\focusA}}$, since the
  introduction rule allows us to use any variable as though it is
  $\focusA$-crisp.
\end{itemize}

\begin{rmk}
  Perhaps surprisingly, the shape of the \rulen{$\sharp$-form} and
  \rulen{$\sharp$-intro} rules is what builds the left-exactness of $\flat$ into
  the theory. This is the case even in ordinary spatial type theory, not
  a feature that only appears in this multi-focus setting. The trick is that
  the promotion operation $\focusA \Gamma$ distributes over the context
  extensions in $\Gamma$ rather than being a `stuck' context former applied to
  $\Gamma$ as a whole. Specifically, when using $\sharp$ to derive crisp
  $\mathsf{Id}$-induction, one applies $\sharp$ to a type
  \[x :: A, y :: A, p :: (x = y) \mid \cdot \yields C \jtype,\] yielding a type
  \[ \cdot \mid x : A, y : A, p : (x = y) \yields \sharp C \jtype.\]
  Internalized, the former context represents the type
  $\dsum{x : \flat A} \dsum{y : \flat A} \flat (x_\flat = y_\flat)$, but
  \rulen{$\sharp$-form} treats it as identical to $\flat \left( \dsum{x : A}
    \dsum{y : A} (x = y) \right)$ when applying adjointness.
\end{rmk}

\begin{rmk}
  In most cases of interest, our commutative idempotent monoid of focuses is
  freely generated by a finite set of basic focuses. In this situation, it
  suffices to provide the $\flat$ and $\sharp$ only for the basic focuses, as
  the remainder can be derived. The
  top focus $\top$ (which semantically corresponds to the entire topos we are
  working in) has both $\flat_{\top} A$ and $\sharp_{\top}A$ canonically
  equivalent to $A$. And given focuses $\focusA$ and $\focusB$, it is quickly
  proven that $\flat_{\focusA\focusB}$ is equivalent to
  $\flat_{\focusA}\flat_{\focusB}$ and similarly for the $\sharp$s.
\end{rmk}

\paragraph{Related Type Theories.}

Besides the original spatial type theory, there are several other
dependent modal type theories that come close to our needs.

The `adjoint type theory' perspective \cite{reed:modal, ls:1var, lsr:multi} was
the guiding principle that led to the original spatial type theory of
\cite{mike:real-cohesive-hott}. Indeed, when instantiated with appropriate mode
theory, the framework of \cite{lsr:multi} reproduces a simply typed version of
the theory presented here. The specific mode theory to be used is a cartesian
monoid with a system of commuting, product-preserving endomorphisms. A
dependently typed variant of adjoint type theory is not yet forthcoming, but we
expect that our dependent type theory would be an instance of it.

An separate line of work on modal type theories is Multimodal Type Theory
\cite{mtt, mtt:tech-report}. In MTT, every mode morphism $\mu$ is reified in the type
theory as a \emph{positive} type former, and each modality $\mathsf{mod}_\mu$
must have a left-adjoint-like context operator written $\lock_\mu$. If we do not
assume the existence of $\shape$, then we are only able to describe $\sharp$ in
this way.

Later work \cite{mtt-fitch} describes a multimodal type theory where each mode
morphism becomes a (more convenient) \emph{negative} type former. The semantic
requirements are even stronger: the functor corresponding to the modality must
be a dependent right-adjoint \cite{drats}, whose left adjoint is itself a
parametric right adjoint. This is too strict even to capture $\sharp$ without
additional assumptions.

In \cite[\S 14]{cavallo:thesis}, an alternative `cohesive type theory' is presented,
using a combination of the above two styles of modal operator. Rather than
working with the endofunctors on the topos of interest, the cohesive setting is
kept as an adjoint quadruple
$\Pi_0 \dashv \mathrm{Disc} \dashv \Gamma \dashv \mathrm{CoDisc}$. A positive
type former is used for $\mathrm{Disc}$ and negative type formers for $\Gamma$
and $\mathrm{CoDisc}$, due to the requirements on having one or two
left-adjoints. It is likely that this could be extended to commuting cohesions,
but the interactions of the various context $\lock_{-}$ operations for the
left-adjoints may be difficult to describe.

The type theory with context structure most formally similar to ours is
ParamDTT~\cite{nvd:parametric-quantifiers, nd:degrees-of-relatedness}, where
variables in the context annotated with a modality indicate a variable under
that modality directly, not its left adjoint. It is from this work that we take
the left-division notation $- \setminus \Gamma$ for the clearing operation on
contexts, which itself has appeared in other guises, for example
\cite{pfenning:intensionality, abel:thesis, abel:polarised-subtyping}. ParamDTT
uses a fixed `mode theory' with three modalities $\{\P, \mathsf{id}, \sharp\}$
equipped with a particular composition law, but it is clear that the rules for
contexts and basic type formers would work equally well for other sets of
modalities. A version of the cohesive $\flat$ can be derived from the `modal
$\Sigma$-type', fixing the second component to be the unit type. There does not
appear to be a way to derive the ordinary (negative) rules for $\sharp$
in ParamDTT.

\section{Specializing a Focus}\label{sec:specializing}
A focus gives a specific axis along which a type may be
spatial. In simplicial cohesion, we have a simplicial focus $\sk \dashv \csk$
and in differential cohesion a differential focus $\flat \dashv \sharp$. But
what makes the simplicial focus \emph{simplicial} and the differential focus
\emph{differential}? In this section, we will investigate two axioms schemes which can determine the
peculiarities of a given focus. In the next section, we will see these axioms in use.

First, we note that with a single focus, type theory with commuting focuses is the
same theory as Shulman's \emph{spatial type theory} in
\cite{mike:real-cohesive-hott}.

\begin{thm}
  Any of the lemmas and theorems proven in \S 3, 4, 5, and 6 of
  \emph{Real Cohesion} \cite{mike:real-cohesive-hott} concerning $\flat$ and
  $\sharp$ and using no axioms are true also of $\flat_{\focusA}$ and
  $\sharp_{\focusA}$ for any fixed focus $\focusA$. Theorems which do involve
  the use of axioms are also valid, so long as the crispness used in those
  axioms is interpreted as $\focusA$-crispness.
\end{thm}
\begin{proof}
  The rules for $\flat_{\focusA}$ and $\sharp_{\focusA}$ specialize to
  Shulman's rules, and therefore his proofs carry through directly.
\end{proof}

Specifically, $\flat_{\focusA}$ is a coreflector and $\sharp_{\focusA}$ is a
monadic modality, both are lex, and $\flat_{\focusA}$ is ($\focusA$-crisply)
left-adjoint to $\sharp_{\focusA}$.

Since adding a focus only expands the rules of the type theory and does not
restrict the application of any of the rules for any of the other focuses, any
of the theorems proven in this section for a single focus will apply when
working with multiple focuses as well.

For the rest of this section, we will work within a single focus $\focusA$, and
for that reason we will drop the annotations by $\focusA$ in our expressions.
For example, we will write $\flat_{\focusA}$ as simply $\flat$, and we will
write $x :_{\focusA} X$ as $x :: X$, following Shulman.

\subsection{Detecting Continuity}

In this section, we will look at an axiom which ties the
liminal sort of ``continuity'' implied by the crisp variables of the type theory
to the concrete continuity of a particular type $G$.

Our axiom will take the form of a lifting property characterizing those crisp
maps which are $\sharp$-modal. As we will show in the upcoming
\cref{thm:sharp.modal.map.lifting}, a crisp map is $\sharp$-modal if
and only if it lifts crisply (in a sense made precise in
\cref{defn:crisp.lifting}) against all of the $\flat$-counits.

\begin{defn}\label{defn:crisp.lifting}
  Let $c :: A \to B$ and $f :: X \to Y$ be crisp
  maps. We say that $c$ \emph{lifts crisply against $f$} if for any
  crisp square as on the left below, there is a unique crisp filler.
\[
  \begin{tikzcd}
    A & X \\
    B & Y
    \arrow["c"', from=1-1, to=2-1]
    \arrow["f", from=1-2, to=2-2]
    \arrow["\forall", from=1-1, to=1-2]
    \arrow["\forall"', from=2-1, to=2-2]
    \arrow["{\exists !}", dashed, from=2-1, to=1-2]
  \end{tikzcd}
  \quad\quad\quad
  \begin{tikzcd}
    {\flat(X^B)} & {\flat(X^A)} \\
    {\flat(Y^B)} & {\flat(Y^A)}
    \arrow["{\flat(\circ f)}", from=1-2, to=2-2]
    \arrow["{\flat(c \circ)}"', from=2-1, to=2-2]
    \arrow["{\flat(\circ f)}"', from=1-1, to=2-1]
    \arrow["{\flat(c \circ)}", from=1-1, to=1-2]
    \arrow["\lrcorner"{anchor=center, pos=0.125}, draw=none, from=1-1, to=2-2]
  \end{tikzcd}
\]
More formally, we write $c \perp_{\flat} f$ for the proposition that the
square on the right is a pullback.
\end{defn}

\begin{thm}\label{thm:sharp.modal.map.lifting}
  A crisp map $f :: X \to Y$ is $\sharp$-modal if and only if for all crisp $A$,
  $(\varepsilon : \flat A \to A) \perp_{\flat} f$.
\end{thm}
\begin{proof}
If $f$ is $\sharp$-modal, then since $\sharp$ is lex, it lifts on the right against all
$\sharp$-equivalences. For any crisp $A$, the
$\flat$-counit $\varepsilon : \flat A \to A$ is
a $\sharp$-equivalence by
\cite[Theorem 6.22]{mike:real-cohesive-hott}. Therefore, the square
\[
  \begin{tikzcd}
    {X^A} & {X^{\flat A}} \\
    {Y^A} & {Y^{\flat A}}
    \arrow["f\circ"', from=1-1, to=2-1]
    \arrow["{\circ\varepsilon}", from=1-1, to=1-2]
    \arrow["{f \circ }", from=1-2, to=2-2]
    \arrow["{\circ\varepsilon}"', from=2-1, to=2-2]
    \arrow["\lrcorner"{anchor=center, pos=0.125}, draw=none, from=1-1, to=2-2]
  \end{tikzcd}
\]
is a pullback, and since $\flat$ preserves crisp pullbacks
(\cite[Theorem 6.10]{mike:real-cohesive-hott}), we see that
$\varepsilon \perp_{\flat} f$.

On the other hand, suppose that $f$ lifts crisply on the right
against all $\flat$-counits. To show that $f$ is $\sharp$-modal, it will
suffice to show that its $\sharp$-naturality square is a pullback.
Let $X \to \sharp X \times_{\sharp Y} Y$ be the gap map of
the $\sharp$-naturality square of $f$, seeking to show that this map
is an equivalence. It suffices to split the gap map over the naturality square, by
the universal property of the pullback. So, consider the crisp square
\[\begin{tikzcd}
    {\flat(\sharp X \times_{\sharp Y} Y)} & X \\
    {\sharp X \times_{\sharp Y} Y} & Y
    \arrow["\snd"', from=2-1, to=2-2]
    \arrow["{\varepsilon}"', from=1-1, to=2-1]
    \arrow["f", from=1-2, to=2-2]
    \arrow["F", from=1-1, to=1-2]
    \arrow["k", dashed, from=2-1, to=1-2]
  \end{tikzcd}
\]
where $F(t) :\defeq (\ind{p^{\flat}}{t}{(\fst p)_{\sharp}})$
is a version of the first projection. To check that the square commutes, it
suffices by $\flat$-induction to give, for crisp
elements $u :: \sharp X$, $y :: Y$, and $p
:: (\sharp f(u) = y^{\sharp})$, a term of type
$f(u_{\sharp}) = y$. But we have crisply that $\sharp f(u) \defeq
\sharp f (u_{\sharp}{}^{\sharp}) =
(f(u_{\sharp}))^{\sharp}$ by the definition of
$\sharp f$, and composing this path with $p$ we know $p' ::
(f(u_{\sharp}))^{\sharp} = y^{\sharp}$. By the
lexness of $\sharp$, we therefore also have $p'' ::
\sharp(f(u_{\sharp}) = y)$, so that the square commutes by $p''{}_{\sharp}$.

By hypothesis, there is a unique crisp map
$k : \sharp X \times_{\sharp Y} Y \to X$ filling this
square. The bottom triangle says precisely that $k$ lives over the second projection. We
will turn the top triangle into a proof that $k$ lives over the first projection.

Let $(u, y , p) : \sharp X \times_{\sharp Y} Y$, seeking to show that
$k(u, y, p)^{\sharp} = u$. This latter type of paths is codiscrete (because
$\sharp X$ is codiscrete), and so when mapping into it we may assume by that $u$
is of the form $x^{\sharp}$, reducing our goal to
$k(x^{\sharp}, y, p)^{\sharp} = x^{\sharp}$. By the lexness of $\sharp$, it
suffices to give an element of $\sharp(k(x^{\sharp}, y, p) = x)$, and for this
it suffices to give an element of $k(x^{\sharp}, y, p) = x$ under the hypotheses
that $x$, $y$, and $p$ are crisp. In this case,
$(x^{\sharp}, y, p)^{\flat} : \flat( \sharp X \times_{\sharp Y} Y)$, and so we
have that $k(x^{\sharp}, y, p) = F((x^{\sharp}, y, p)^{\flat})$ by the upper
triangle. But by definition,
$F((x^{\sharp}, y, p)^{\flat}) \defeq x^{\sharp}{}_{\sharp} \defeq x$, so that we
have succeeded in giving the desired identification.

We have shown that $k$ lives over the naturality square; now we need to show
that it splits the gap map $X \to \sharp X
\times_{\sharp Y} Y$. To that end, consider the following diagram:
\[
  \begin{tikzcd}
    {\flat X} & {\flat(\sharp X \times_{\sharp Y} Y)} & X \\
    X & {\sharp X \times_{\sharp Y} Y} & Y
    \arrow["{\type{gap}}" swap, from=2-1, to=2-2]
    \arrow["{\flat\type{gap}}" swap, from=1-1, to=1-2]
    \arrow[from=1-2, to=1-3]
    \arrow["{\varepsilon}", bend left, from=1-1, to=1-3]
    \arrow["{\varepsilon}"', from=1-1, to=2-1]
    \arrow[from=2-2, to=2-3]
    \arrow["f", from=1-3, to=2-3]
    \arrow["f"', bend right, from=2-1, to=2-3]
    \arrow[from=1-2, to=2-2]
    \arrow["k", dashed, from=2-2, to=1-3]
  \end{tikzcd}
\]
  Showing that the diagram commutes as drawn follows easily by
  $\flat$-induction. We then have two crisp fillers of the outer square:
  first we have $\id_X : X \to X$ and $k \circ \type{gap} : X \to X$. By the
  uniqueness of crisp fillers, we conclude that these must be identical.
\end{proof}

Knowing that the crisp $\sharp$-modal maps may be characterized by lifting
crisply against $\flat$-counits suggests that we could axiomatize the particular
qualities of $\sharp$ by restricting the class of $\flat$-counits which it
suffices to lift against. To that end, we make the following definition.

\begin{defn}
  Let $G :: I \to \Type$ be a crisp type family indexed by a
  $\flat$-modal and inhabited type $I$. We say that
  $G$ \emph{detects continuity} when, for every crisp map $f :: X \to Y$,
  \[
  \begin{tikzcd}
    \left\{ \text{$f$ is $\sharp$-modal} \right\} \ar[r, leftrightarrow] & {\left\{   \makecell[c]{(\varepsilon : \flat G_i \to G_i) \perp_{\flat} f \\ \text{for all $i :: I$}} \right\}}
  \end{tikzcd}
  \]
\end{defn}

\begin{rmk}
Thinking externally, it is straightforward to see that any family $G_i$ which
generates a local topos $\Ea$ in question under colimits will detect continuity
for the focus given by the terminal map of toposes. This is because $\flat$, as a left adjoint, commutes with all
colimits; therefore the problem of lifting against the $\flat$-counit of any
object of $\Ea$ can be reduced to that of lifting against the $\flat$-counits of
the generators $G_i$.
\end{rmk}

\begin{lem}
  A crisp type $X$ is $\sharp$-modal if and only if $\flat(A \to X) \to
  \flat(\flat A \to X)$ is an equivalence for all crisp types $A$, and if $G$
  detects continuity then it suffices to check for each $G_i$.
\end{lem}
\begin{proof}
  When $f : X \to 1$, the square defining crisp lifting is a pullback iff the
  top map is an equivalence.
\end{proof}

If a family detects continuity, then it is a separating family for
crisp maps in the following precise sense.

\begin{thm}\label{thm:equivalence.on.generators.gives.equivalence}
Suppose that $G :: I \to \Type$ detects continuity. Let $f
:: X \to Y$ be a crisp map for which $\flat f :
\flat X \to \flat Y$ is an equivalence and for all $i :: I$, the induced
map $\flat(G_i \to X) \to \flat(G_i \to Y)$ given by
post-composing with $f$ is an equivalence. Then $f$ is an equivalence.
\end{thm}
\begin{proof}
First, note that $f$ is a $\sharp$-equivalence since it is by
hypothesis a $\flat$-equivalence. It therefore suffices to show that
$f$ is $\sharp$-modal, which by the assumption that $G$ detects
continuity means showing that $f$ lifts crisply against
all $\flat$-counits $\flat G_i \to G_i$ for $i
:: I$. Consider the following diagram:
\[
  \begin{tikzcd}
    {\flat(X^{G_i})} & {\flat(X^{\flat G_i})} & {\flat(\sharp X^{G_i})} \\
    {\flat(Y^{G_i})} & {\flat(Y^{\flat G_i})} & {\flat(\sharp Y^{G_i})}
    \arrow[from=1-1, to=1-2]
    \arrow[from=2-1, to=2-2]
    \arrow["{\flat(f\circ)}"', from=1-1, to=2-1]
    \arrow["{\flat(f\circ)}", from=1-2, to=2-2]
    \arrow["\sim", from=1-2, to=1-3]
    \arrow["\sim", from=2-2, to=2-3]
    \arrow["{\flat(\sharp f\circ)}", from=1-3, to=2-3]
  \end{tikzcd}
\]
The square on the left is the one we need to show is a pullback. For this, it
will suffice to show that the middle map
$\flat(f\circ) : \flat(X^{\flat G_i}) \to \flat(Y^{\flat G_i}) $ is an
equivalence, since the leftmost vertical map is an equivalence by hypothesis and
any square with two sides equivalences is a pullback.

For that, the middle vertical map is equivalent by the adjunction
$\flat \dashv \sharp$ to the rightmost vertical map (\cite[Corollary
6.26]{mike:real-cohesive-hott}). But the rightmost vertical map is an equivalence
because it is post-composition by the equivalence $\sharp f$.
\end{proof}

\subsection{Detecting Connectivity}

A focus is said to be \emph{cohesive} if $\flat$ has a
further left adjoint $\shape$ which is itself a modality:
\[
\flat(\shape A \to X) \equiv \flat(A \to
\flat X).
\]

This adjunction only determines $\shape$ for crisp types. It
is better to define $\shape$ by nullifying a family of objects; then
$\shape$ is determined for all types (of any size). To this end, we
make the following definition.

\begin{defn}\label{defn:detect.connectivity}
  Let $G :: I \to \Type$ be a crisp type family indexed by a
  $\flat$-modal type. We say that $G$ \emph{detects connectivity} when, for any crisp type $X$,
  \[
  \begin{tikzcd}
    \left\{ \text{$X$ is $\flat$-modal} \right\} \ar[r, leftrightarrow] & {\left\{   \makecell[c]{\text{$X$ is $G_i$-null}  \\ \text{for all $i :: I$}} \right\}}
  \end{tikzcd}
  \]

  If $G$ detects connectivity, then $\shape$ is defined to be
  nullification at the family $G_i$.
\end{defn}

\begin{rmk}
  In \emph{Real Cohesion} \cite{mike:real-cohesive-hott}, the assertion that a given
  family $G$ detects connectivity is known as Axiom C0. In
  \cite[Definition 5.2.48]{urs:diff-coh}, a single object with this property is
  said to `exhibit the cohesion'.
\end{rmk}

If there is a family $G$ which detects connectivity, then we say that
the focus is \emph{cohesive}. This is justified by the following theorem, which
we may import directly from \emph{Real Cohesion} \cite{mike:real-cohesive-hott}.
\begin{thm}
Suppose that $G$ detects connectivity.  Then a crisp
type is $\shape$-modal if and only if it is $\flat$-modal,
and furthermore $\shape$ is crisply left-adjoint to $\flat$:
\[
\flat(\shape A \to X) \to \flat (A \to \flat X)
\]
\end{thm}
\begin{proof}
This is \cite[Theorem 9.15]{mike:real-cohesive-hott}.
\end{proof}

\section{Examples of Focuses}\label{sec:single.examples}

To keep the various operators visually distinct, we will use
completely different symbols for each focus we are interested in. The
rules governing the type formers are unchanged.

\begin{itemize}
\item $\shape \dashv \flat \dashv \sharp$ denotes \emph{real cohesion}, where a
  set of real numbers (possible the Dedekind reals or an axiomatically asserted
  set of ``smooth reals'') detects connectivity.

\item $\geomreal \dashv \sk \dashv \csk$ denotes \emph{simplicial cohesion},
  where the (axiomatically asserted) $1$-simplex $\Delta[1]$ detects connectivity.

\item $\orbishape \dashv \orbiflat \dashv \orbisharp$ denotes \emph{global
    equivariant cohesion}, where connectivity is detected by $\orbisharp \B G$ for
  finite groups $G$. This notational convention follows Sati and Schreiber
  \cite{sati-schreiber:proper-orbifold-cohomology}.

\item Various \emph{topological toposes} exhibit spatial type theory, with
  $\flat \dashv \sharp$ retopologizing types with the discrete and codiscrete
  topologies respectively. In particular, Johnstone's topological topos has a
  focus whose continuity is detected by the walking convergent sequence
  $\Nb_{\infty}$, which may be constructed as the set of monotone functions
  $\Nb \to \{0, 1\}$.
\end{itemize}

\subsection{Real Cohesions}

In \emph{Real Cohesion} \cite{mike:real-cohesive-hott}, Shulman
gives the axiom $\Rb\flat$ which states that a crisp type is $\flat$-modal if
and only if it is $\Rb_D$-null, where $\Rb_D$ is the set of Dedekind cuts. In the
terminology of \cref{defn:detect.connectivity}, this says that $\Rb_D$ detects
continuous real connectivity.

\begin{axiom}[Continuous Real Cohesion]
We assume that $\Rb_D$ detects continuous real connectivity, and also Shulman's Axiom
T: For every $x : \Rb_D$, the proposition $(x > 0)$ is $\sharp$-modal.
\end{axiom}

\begin{rmk}
Though Shulman does not consider this axiom, we may also add the assumption that the
family $\Rb_D^n$ detects continuous real continuity. Using this assumption, we
may internalize the arguments of Example 8.33 of \cite{mike:real-cohesive-hott} to
show that the mysterious Axiom T follows from the proposition that if $f
:: \Rb_D^n \to \Rb_D$ is crisp and $f(x) > 0$ for any crisp $x :: \Rb_D^n$,
then in fact $f(x) > 0$ for all (not necessarily crisp) $x :
\Rb_D^n$. Since, by
Corollary 8.28 of
\cite{mike:real-cohesive-hott} (assuming the crisp LEM or the
axiom of countable choice), any crisp Dedekind real is
a Cauchy real, we are equivalently
asking if a function $f : \Rb_D \to
\Rb_D$ is positive on all Cauchy
reals, is it always positive. This
seems obvious, but as Shulman notes
in Example 8.34, this obvious
statement is not always true; though
assuming countable choice it is
likely provable.
\end{rmk}

There are continuous but non-differentiable functions $f : \Rb_D \to \Rb_D$. If
we want to work in a topos where the types have a \emph{smooth} structure
instead of just a continuous structure, then we must work with a type of
\emph{smooth} reals $\Rb_S$. The most common way to axiomatize the type of
smooth reals is using the Kock-Lawvere axiom and the other axioms of synthetic
differential geometry. See, e.g. Section 4.1 of \cite{jaz:orbifolds} for a list
of these axioms. In any case, if $\Rb_S$ is a type of smooth reals, then we will
take differential cohesion to mean that $\Rb_S$ detects connectivity.

\begin{axiom}[Differential Real Cohesion]
If $\Rb_S$ is a type of smooth reals (say, from synthetic differential
geometry), then we assume that $\Rb_S$ detects differential connectivity.
\end{axiom}

\subsection{Simplicial Cohesion}

There is a well known difficulty in describing simplicial types in ordinary
homotopy type theory --- the infinite amount of coherence data is
difficult to describe formally given the tools of type theory. This difficulty
has led to extensions of type theory such as two-level type theories \cite{ack:hott-strict-equality, two-level-tt, univalence-principle}
which augment HoTT with strict equalities which can be use to define simplicial
homotopy types satisfying the simplicial identities strictly, bypassing the
problematic tower of coherences.

Another approach to avoiding simplicial difficulties is to simply interpret type
theory into a topos of \emph{simplicial} homotopy types, rather than mere homotopy
types. This is the approach taken by Riehl and Shulman in \cite{rs:synthetic-cats}, where they
present a type theory that makes every type into a simplicial
type and has as primitives the simplices $\Delta[n]$, so that a simplex in a
type $A$ is a function $\sigma : \Delta[n] \to A$.

In this section we will also work with simplicial homotopy types, but by
different means to Riehl and Shulman's type theory. Instead, we will describe
\emph{simplicial cohesion}, with adjoint modalities
$\geomreal{} \dashv \sk \dashv \csk$. These are defined semantically as follows:

\begin{itemize}
        \item The (``simplicial flat'') $0$-skeletal comodality $X \mapsto \sk X$ sends a simplicial type to its $0$-skeleton:
        \[
        \begin{tikzcd}[row sep = small]
          \overset{\raisebox{4pt}{\vdots}}{X_{2}}
          \ar[d, leftarrow, shift left = 10]
          \ar[d, shift left = 15]
          \ar[d, shift left = 5]
          \ar[d, leftarrow]
          \ar[d, shift right = 5]
          \ar[d, leftarrow, shift right = 10]
          \ar[d, shift right = 15]\\
          X_{1}
          \ar[d, shift left = 5]
          \ar[d, leftarrow]
          \ar[d, shift right = 5] \\
          X_{0}
        \end{tikzcd}
        \quad\quad\xmapsto{\quad\sk\quad} \quad\quad
        \begin{tikzcd}[row sep = small]
          \overset{\raisebox{4pt}{\vdots}}{X_{0}}\ar[d, leftarrow, shift left = 10]\ar[d, shift left = 15] \ar[d, shift left = 5] \ar[d, leftarrow] \ar[d, shift right = 5]\ar[d, leftarrow, shift right = 10]\ar[d, shift right = 15]\\
          X_{0}\ar[d, shift left = 5] \ar[d, leftarrow] \ar[d, shift right = 5] \\
          X_{0}
        \end{tikzcd}
        \]
       \item The (``simplicial sharp'') $0$-coskeletal modality $X \mapsto \csk X$ sends a simplicial type to its $0$-coskeleton:
        \[
        \begin{tikzcd}[row sep = small]
          \overset{\raisebox{4pt}{\vdots}}{X_{2}}\ar[d, leftarrow, shift left = 10]\ar[d, shift left = 15] \ar[d, shift left = 5] \ar[d, leftarrow] \ar[d, shift right = 5]\ar[d, leftarrow, shift right = 10]\ar[d, shift right = 15]\\
          X_{1}\ar[d, shift left = 5] \ar[d, leftarrow] \ar[d, shift right = 5] \\
          X_{0}
        \end{tikzcd}
        \quad\quad\xmapsto{\quad\csk\quad} \quad\quad
        \begin{tikzcd}[row sep = small]
          \overset{\raisebox{4pt}{\vdots}} {X_{0} \times X_{0} \times X_{0}}\ar[d, leftarrow, shift left = 10]\ar[d, shift left = 15] \ar[d, shift left = 5] \ar[d, leftarrow] \ar[d, shift right = 5]\ar[d, leftarrow, shift right = 10]\ar[d, shift right = 15]\\
          X_{0} \times X_{0}\ar[d, shift left = 5] \ar[d, leftarrow] \ar[d, shift right = 5] \\
          X_{0}
        \end{tikzcd}
        \]

        \item The (``simplicial shape'') realization modality $X \mapsto \geomreal{X}$ sends a simplicial type to its realization (or homotopy colimit), considered as a $0$-skeletal simplicial type:
        \[
        \begin{tikzcd}[row sep = small]
          \overset{\raisebox{4pt}{\vdots}}{X_{2}}\ar[d, leftarrow, shift left = 10]\ar[d, shift left = 15] \ar[d, shift left = 5] \ar[d, leftarrow] \ar[d, shift right = 5]\ar[d, leftarrow, shift right = 10]\ar[d, shift right = 15]\\
          X_{1}\ar[d, shift left = 5] \ar[d, leftarrow] \ar[d, shift right = 5] \\
          X_{0}
        \end{tikzcd}
        \quad\quad\xmapsto{\quad\geomreal{\cdot}\quad} \quad\quad
        \begin{tikzcd}[row sep = small]
          \overset{\raisebox{4pt}{\vdots}}{\colim X_{n}} \ar[d, leftarrow, shift left = 10]\ar[d, shift left = 15] \ar[d, shift left = 5] \ar[d, leftarrow] \ar[d, shift right = 5]\ar[d, leftarrow, shift right = 10]\ar[d, shift right = 15]\\
       \colim X_{n}  \ar[d, shift left = 5] \ar[d, leftarrow] \ar[d, shift right = 5] \\
         \colim X_{n}
        \end{tikzcd}
        \]
\end{itemize}

Because simplicial sets are the
classifying ($0$-)topos for strict intervals (totally ordered sets with distinct
top and bottom elements) \cite{wraith:generic-interval}, and since the $\infty$-topos of simplicial
homotopy types is the \emph{enveloping $\infty$-topos}\footnote{The enveloping
  $\infty$-topos of a topos is its free (homotopy) cocompletion, fixing existing
  homotopy colimits.} of simplicial sets \cite{anel:enveloping-topos}, we
may assume the existence of an interval $\Delta[1]$ to make sure that our type
theory is interpreted in an $\infty$-topos equipped with a geometric morphism to
$\Sa^{\Delta\op}$. We may then define the $n$-simplex $\Delta[n]$ to be the
$n$-length increasing sequences in $\Delta[1]$, and define an $n$-simplex in a
type $X$ to be a map $\Delta[n] \to X$.\

\begin{axiom}[Simplicial Axioms]
  We presume that $\Delta[1]$ is a total order with distinct top and bottom elements which we call $1$ and $0$ respectively. Explicitly, this means that we have elements $0,\, 1 : \Delta[1]$ and a proposition $x \leq y : \Prop$ for $x,\, y : \Delta[1]$. This order must satisfy the following axioms:
  \begin{enumerate}
  \item For all $x$, $x \leq x$.
  \item For all $x$, $y$, and $z$, if $x \leq y$ and $y \leq z$ then $x \leq z$.
  \item For all $x$ and $y$, if $x \leq y$ and $y \leq x$, then $x = y$.
  \item For all $x$, $y$, either $x \leq y$ or $y \leq x$.
  \item For all $x$, $0 \leq x$ and $x \leq 1$.
  \item $0 \neq 1$.
  \end{enumerate}
  From these axioms, we may define the other simplices $\Delta[n]$
  to be the chains of length $n$ in $\Delta[1]$:
  $$\Delta[n] :\defeq \{\vec{x} : \Delta[1]^{n} \mid x_{1} \leq x_{2} \leq \cdots \leq x_{n}\}.$$

  We also assume the following:
  \begin{itemize}
  \item (Axiom $\Delta\sk$) $\Delta[1]$ detects simplicial connectivity: a
    simplicially crisp type $X$ is $0$-skeletal if and only if every map
    $\Delta[1] \to X$ is constant.
  \item The family $\Delta[-] : \Nb \to\Type$ detects simplicial continuity.
  \item (Axiom $\partial\Delta$) For $i : \Delta[1]$, we have $\csk((i = 0) \vee (i = 1))$.
  \item Each $\Delta[n]$ is crisply projective. That is, for a simplicially crisp $E :
      \Delta[n] \to \Type$, we have a map
      \[
        \sk(\dprod{i : \Delta[n]} \exists E_i) \to \exists \sk(\dprod{i : \Delta[n]} E_i).
      \]
      As there is an obvious map the other way, this map is an equivalence.
  \end{itemize}
\end{axiom}

Let us quickly set the stage by proving that the $n$-simplices have trivial
geometric realization and
the $0$-skeleton of the $n$-simplex
$\Delta[n]$ is the ordinal $[n] \defeq \{0, \ldots, n\}$.

\begin{lem}
The order $\Delta[1]$ has finite meets and joins, and they distribute over
each other. Moreover, the inclusion $\{0, 1\} \hookrightarrow \Delta[1]$ is an
inclusion of lattices.
\end{lem}
\begin{proof}
Suppose that $x, y : \Delta[1]$. Then either $x \leq y$ or $y \leq x$. In the
former case, define $x \wedge y :\defeq x$ and $x \vee y :\defeq y$, and in the
latter case $x \wedge y :\defeq y$ and $x \vee y
:\defeq x$. If both hold, then $x = y$ and the definitions agree.

If $x, y, z : \Delta[1]$, then these three may find themselves in any of $6$
orderings. One may then check that in each of these cases, meets distribute over
joins and vice versa. For example, supposing that $x \leq y \leq z$, then $x
\wedge (y \vee z) = x \wedge z = x$, while $(x \wedge y) \vee (x \wedge z) = x
\vee x = x$.
\end{proof}

\begin{lem}\label{lem:delta.n.retract}
  The $n$-simplex $\Delta[n]$ is a retract of the $n$-cube $\Delta[1]^n$.
  Moreover, the inclusion $[n] \hookrightarrow \Delta[n]$ given by
\[
  i \mapsto
  0 \leq \cdots \leq 0 \leq \overbrace{1 \leq \cdots \leq 1}^{\mbox{$i$ times}}
\]
is a retract of the inclusion $\{0,1\}^n \to \Delta[1]^n$.
\end{lem}
\begin{proof}
Given $x_1,\ldots, x_n : \Delta[1]$, define $m_1 :\defeq \bigwedge_{i : \ord{n}}
x_i$ and let $i_1 : \ord{n}$ be its index, then $m_2 :\defeq \bigwedge_{\ord{n} \setminus \{m_1\}} x_i$, and so on. Note that $m_1 \leq m_2 \leq \cdots
\leq m_n$, so that $m : \Delta[n]$. Finally, if the $x_i$ were already in
increasing order, then $m_i = x_i$, showing that this is a retract.

We note that this retract argument works just as well on $\{0,1\}^n \to [n]$, if
we identify $[n]$ with the subset $\{\vec{x} : \{0, 1\}^n \mid x_1 \leq \cdots
\leq x_n\}$ of increasing sequences. Since it only makes use of the lattice
structure of $\{0, 1\}^n$ and $\Delta[1]^n$, and the inclusion is a lattice
homomorphism, we conclude that the necessary squares commute.
\end{proof}

\begin{thm}\label{lem:delta.n.contractible}
  The $n$-simplex has trivial realization: $\geomreal \Delta[n] = \ast$.
\end{thm}
\begin{proof}
The realization $\geomreal \Delta[n]$ is a retract of the realization $\geomreal
\Delta[1]^n$, and this is contractible since $\Delta[1]$ detects simplicial connectivity.
\end{proof}

\begin{thm}\label{thm:skeleton.simplex}
  The inclusion $[n] \hookrightarrow \Delta[n]$ given by
  \[
    i \mapsto
    0 \leq \cdots \leq 0 \leq \overbrace{1 \leq \cdots \leq 1}^{\mbox{$i$ times}}
  \] is a $\sk$-equivalence, showing that $\sk\Delta[n] \simeq [n]$.
\end{thm}
\begin{proof}
We will show that $[1] \hookrightarrow \Delta[1]$ is a $\sk$-equivalence. This
will imply that $[1]^n \hookrightarrow \Delta[1]^n$ is a $\sk$-equivalence;
since $[n] \hookrightarrow \Delta[n]$ is a retract of this, we may conclude that
it is a $\sk$-equivalence as well.

Since this inclusion $\{0, 1\} \hookrightarrow \Delta[1]$ is simplicially crisp, to show that it is a $\sk$-counit it
will suffice to show that it is a $\csk$-equivalence. We therefore need an
inverse $\csk \Delta[1] \to \csk\{0, 1\}$. Since the codomain is $0$-coskeletal,
it suffices to define this map on $\Delta[1]$. So let $i : \Delta[1]$, seeking
$\csk\{0, 1\}$. By Axiom $\partial\Delta$, we have $\csk((i = 0) \vee (i =
1))$, and since our goal is $0$-coskeletal, we may assume that $i = 0$ or $i =
1$. If $i = 0$, then we send it to $0^{\csk}$, if $i = 1$, then we send it to $1^{\csk}$.

To show that this map is the inverse of $\csk\{0, 1\} \to \csk\Delta[1]$, we may
appeal to the fact that identities in a modal type are modal, and so we may
remove the $\csk$ around $\csk((i=0) \vee (i = 1))$ and check that the maps
invert each other on these elements, which they clearly do.
\end{proof}

We can also define the type of $n$-simplices in a simplicially
crisp type, and prove a few elementary lemmas concerning the $n$-simplices of types.

\begin{defn}
  Let $X$ be a simplicially crisp type. Then define the type $X_n$ of
  $n$-simplices in $X$ as
  \[
    X_n :\defeq \sk(\Delta[n] \to X).
  \]
  If $f : X \to Y$ is a simplicially crisp map, then it induces a map $f_n : X_n
  \to Y_n$ by post-composition.
\end{defn}

\begin{lem}
Let $f : X \to Y$ be a simplicially crisp map. If $f_n : X_n\to Y_n$ is an
equivalence for all $n$, then $f$ is an equivalence.
\end{lem}
\begin{proof}
This is a special case of
\cref{thm:equivalence.on.generators.gives.equivalence}, noting that $X_0 \simeq
\sk X$.
\end{proof}

\begin{lem}
  Let $f : X \to Y$ be a simplicially crisp map. Then for a crisp $y : \Delta[n]
  \to Y$, we have
  \[
\fib_{f_n}(y^{\sk}) \equiv {\sk(\dprod{i : \Delta[n]} \fib_f(y(i)))}.
  \]
\end{lem}
\begin{proof}
  We compute:
  \begin{align*}
    \fib_{f_n}(y^{\sk})
    &\defeq \dsum{x : X_n} (f_n x = y^{\sk}) \\
    &\defeq \dsum{x : \sk(\Delta[n] \to X)} \ind{x}{\tau^{\sk}}{(f \circ \tau)^{\sk} = y^{\sk}}\\
    &\equiv \dsum{x : \sk(\Delta[n] \to X)} \ind{x}{\tau^{\sk}}{\sk(f \circ \tau = y)} \\
    &\equiv \sk(\dsum{x : \Delta[n] \to X} (f \circ x = y))\\
    &\equiv \sk(\dsum{x : \Delta[n] \to X} (\dprod{i : \Delta[n]} (f(x(i)) = y(i))))\\
    &\equiv \sk(\dprod{i : \Delta[n]} \fib_f(y(i))).
  \end{align*}
\end{proof}

\begin{lem}
Let $f : X \to Y$ be a simplicially crisp map. Then $(\im f)_n \simeq \im f_n$.
\end{lem}
\begin{proof}
  We use the projectivity of the simplices.\footnote{This lemma is in fact
    equivalent to assuming the projectivity of the simplices.}
  \begin{align*}
    (\im f)_n &\defeq \sk(\Delta[n] \to \dsum{y : Y} \exists \fib_f(y)) \\
              &\equiv \dsum{y : Y_n} \ind{y}{\sigma^{\sk}}{\sk(\dprod{i : \Delta[n]} \exists \fib_f(\sigma i))} \\
              &\equiv \dsum{y : Y_n} \ind{y}{\sigma^{\sk}}{\exists \sk(\dprod{i : \Delta[n]} \fib_f(\sigma i))} \\
              &\equiv \dsum{y : Y_n} \exists\fib_{f_n}(y)\\
                &\defeq \im f_n.
  \end{align*}
\end{proof}

The definition of the $n$-simplices that we gave above is simple,
but it is not that straightforward to see that it is functorial in the
ordinal $[n]$. We can give an alternative definition of the
$n$-simplices which makes the functoriality evident.

\begin{defn}
  Let $\type{Interval}$ denote the category of \emph{intervals}: totally ordered sets with distinct top and bottom. The maps of $\type{Interval}$ are the monotone functions preserving top and bottom.

  Let $\type{FinOrd}_{+}$ denote the category of finite inhabited ordinals and order preserving maps between them --- the usual ``simplex category''. We denote by $[n]$ the ordinal $\{0, \ldots, n\}$.
\end{defn}

We will need a standard reformulation of the category of finite ordinals in
terms of intervals (see e.g. \cite[\S VIII.7]{maclane-moerdijk:book})
\begin{lem}
  There is a contravariant, fully faithful functor $\iota : \type{FinOrd}_{+}\op \to \type{Interval}$ sending $[n]$ to $[n + 1]$ with top element $n + 1$ and bottom element $0$. To a map $f : [n] \to [m]$, we define $\iota f : [m+1] \to [n+1]$ by
  \[
    \iota f (i) :\defeq \begin{cases} \min\{j \mid i \leq f(j) \} \\
                        n+1 &\mbox{if no such minimum exists.}\end{cases}
  \]
  Conversely, to a monotone map $g : [m+1] \to [n+1]$ preserving top and bottom, we define $\iota\inv g : [n] \to [m]$ by the dual formula
  \[
(\iota^{-1} g)(j) :\defeq \max\{i \mid g(j) \leq i\}.
  \]
\end{lem}

We may now define the $n$-simplices in a way which makes clear their functoriality in the category of finite inhabited ordinals.
\begin{defn}
  We define the $n$-simplex $\Delta[n]$ to be
  \[
\Delta[n] :\defeq \type{Interval}(\iota[n], \Delta[1]).
  \]
  Therefore, $\Delta : \type{FinOrd}_{+} \to \Set$ gives a functor from finite inhabited ordinals to the category of sets, where $\Delta(f) : \Delta[n] \to \Delta[m]$ is given by precomposing with $\iota f : [m+1] \to [n + 1]$.
\end{defn}

Noting that $[n] \cong \type{Interval}(\iota[n], [1])$ by the fully-faithfulness
of $\iota$, the inclusion of top and bottom elements $[1] \hookrightarrow
\Delta[1]$ induces a natural inclusion $[n] \hookrightarrow \Delta[n]$ by
post-composition. As we saw in \cref{thm:skeleton.simplex}, these inclusions are $\sk$-counits.

\subsubsection{The \v{C}ech Complex}

The $0$-coskeleton modality $\csk$ is useful for working in simplicial cohesion
since it enables us to give an easy construction of the \v{C}ech complex of a
map $f : X \to Y$ between $0$-skeletal types. The \v{C}ech complex of such a map
is, externally speaking, the simplicial type formed by repeatedly pulling back $f$ along itself:
\[
  \check{\type{C}}(f) :\defeq \quad
  \begin{tikzcd}
    \overset{\raisebox{4pt}{\vdots}}{ X \times_{Y} X \times_{Y} X }\ar[d, leftarrow, shift left = 10]\ar[d, shift left = 15] \ar[d, shift left = 5] \ar[d, leftarrow] \ar[d, shift right = 5]\ar[d, leftarrow, shift right = 10]\ar[d, shift right = 15]\\
    X \times_{Y} X\ar[d, shift left = 5] \ar[d, leftarrow] \ar[d, shift right = 5] \\
    X
  \end{tikzcd}
\]

\begin{defn}
  Let $f : X \to Y$ be a map. The \emph{\v{C}ech complex} $\check{\type{C}}(f)$ of $f$ is defined to be its $\csk$-image:
  $$\check{\type{C}}(f) :\defeq \dsum{y : Y} \csk(\dsum{x : X}(fx = y)).$$
  \end{defn}

We will justify this definition by calculating the type of $n$-simplices of
$\check{\type{C}}(f)$ when both $X$ and $Y$ are $0$-skeletal.
  \begin{prop}\label{prop:cech.nerve}
    Let $f : X \to Y$ be a simplicially crisp map between $0$-skeletal types. Then
    $$\check{\type{C}}(f)_{n} \equiv X \times_{Y} \cdots \times_{Y} X \equiv \dsum{y : Y} (\dsum{x : X} (f x = y))^{n + 1}$$
    is the $(n+1)$-fold pullback of $f$ along itself.
 \end{prop}
 \begin{proof}
   We calculate:
   \begin{align*}
     \check{\type{C}}(f)_{n} &:\defeq \sk(\Delta[n] \to \check{\type{C}}(f)) \\
                              &\phantom{:}\defeq \sk(\Delta[n] \to \dsum{y : Y} \csk(\dsum{x : X}(fx = y))) \\
                             &\phantom{:}\simeq \sk(\dsum{\sigma : \Delta[n] \to Y} (\dprod{i : \Delta[n]} \csk(\dsum{x : X}(f x = \sigma i)))) \\
     \intertext{Since $Y$ is $0$-skeletal, any map $\Delta[n] \to Y$ is constant, so we may continue:}
                              &\simeq \sk(\dsum{y : Y} (\Delta[n] \to \csk(\dsum{x : X} (fx = y))))\\
     \intertext{Since, by \cref{thm:skeleton.simplex}, $\sk\Delta[n] = [n]$, we may use the adjointness of $\sk$ and $\csk$ to continue:}
&\simeq \sk(\dsum{y : Y} \csk([n] \to \dsum{x : X}(f x = y))) \\
     \intertext{Now, we may use Lemma 6.8 of \cite{mike:real-cohesive-hott} to pass the $\sk$ into the pair type, and then use that \(\sk \csk = \sk\) to continue: }
                              &\simeq (\dsum{u : \sk Y}\ind{y^{\sk}}{u}{ \sk([n] \to \dsum{x : X}(f x = y)))} \\
     \intertext{However, all types involved are already 0-skeletal, so we may remove the \(\sk\)s:}
                              &\simeq (\dsum{y : Y} ([n] \to \dsum{x : X}(f x = y))) \\
     &\equiv \dsum{y : Y} (\dsum{x : X} (f x = y))^{n + 1}
    \end{align*}
    This last type is the $(n+1)$-fold pullback of $f$ along itself, displayed in terms of its diagonal map to $Y$.
\end{proof}

We can prove modally that the realization of the \v{C}ech nerve of a map $f : X
\to Y$ is the image $\im f$ of $f$. This follows from Theorem 10.2 of
\emph{Real Cohesion} \cite{mike:real-cohesive-hott}.
\begin{thm}\label{thm:re.csk.trivial}
If $A$ is $0$-coskeletal, then $\geomreal A$ is a proposition. As a corollary,
$\geomreal (\csk X) \simeq \trunc{X}$.
\end{thm}
\begin{proof}
In \cite{mike:real-cohesive-hott}, this theorem is said to rely on the crisp Law
of Excluded Middle. However a glance at the proof reveals that this assumption
is only used to assume the decidable equality of $\sk \Delta[1]$. Since we know that $\sk \Delta[1] \equiv \{0,
1\}$ has decidable equality, the proof goes through without assuming crisp LEM.
\end{proof}

\begin{thm}
For a map $f : X \to Y$, the realization $\geomreal\check{\type{C}}(f)$ of the
\v{C}ech nerve is the
realization $\geomreal\im f$ of the image of $f$. If furthermore $Y$ is
$0$-skeletal, then $\geomreal\check{\type{C}}(f) \equiv \im f$.
\end{thm}
\begin{proof}
  We compute:
  \begin{align*}
    \geomreal \check{\type{C}}(f)
    &\defeq \geomreal(\dsum{y : Y} \csk\fib_f(y)) \\
    &\equiv \geomreal(\dsum{y : Y} \geomreal\csk\fib_f(y)) \\
    &\equiv \geomreal(\dsum{y : Y} \trunc{\fib_f(y)}) \\
    &\defeq \geomreal \im f.
  \end{align*}
Now if $Y$ is $0$-skeletal, then by Lemma 8.17 of \cite{mike:real-cohesive-hott},
$\im f$ is also $0$-skeletal, since it is a subtype of a $0$-skeletal type.
Therefore, $\geomreal \im f \equiv \im f$, so that in total $\geomreal
\check{\type{C}}(f) \equiv \im f$.
\end{proof}

As an application of \v{C}ech nerves, we can see how to extract coherence
data for higher groups from their deloopings. If we take the \v{C}ech nerve of
the inclusion $\pt_{\B G} : \ast \to \B G$ of the base point of the delooping of
$G$, we recover a simplicial type whose simplicial identities give coherences
for the multiplication of $G$.

\begin{prop}
  Let $G$ be a crisp, $0$-skeletal higher group --- a 0-skeletal type identified with the
  loops of a pointed, $0$-connected type $\B G$. Then
  \[
\check{\type{C}}(\pt_{\B G})_{n} \simeq G^{n}.
  \]
  Furthermore, $d_{1} : \check{\type{C}}(\pt_{\B G})_{2} \to\check{\type{C}}(\pt_{\B G})_{1}$ is the product of the projections $d_{0}$ and $d_{2} : G^{2} \to G$.
\end{prop}
\begin{proof}
  By \cref{prop:cech.nerve}, we know that
\begin{align*}
  \check{\type{C}}(\pt_{\B G})_{n} &\simeq \dsum{e : \B G} (\dsum{x : \ast} (\pt_{\B G} \ast = e))^{n+1}\\
                                   &\simeq \dsum{e : \B G} (\pt_{\B G} = e)^{n+1}\\
  &\simeq \dsum{e : \B G} (\pt_{\B G} = e) \times (\pt_{\B G} = e)^{n}\\
                                   &\simeq (\pt_{\B G} = \pt_{\B G})^{n} \\
  &= G^{n}.
  \end{align*}
In the second to last step, we contract $\dsum{e : \B G} (\pt_{\B G} = e)$.

Now, $d_{i} : \check{\type{C}}(\pt_{\B G})_{2} \to\check{\type{C}}(\pt_{\B G})_{1}$ is given by forgetting the $i^{\text{th}}$ component of the list $(e, (a, b, c)) : \dsum{e : \B G} (\pt_{\B G} = e)^{n+1}$. Therefore,
\begin{align*}
d_{0}(e, (a, b, c)) &= (e, (b, c))\\
d_{1}(e, (a, b, c)) &= (e, (a, c))\\
d_{2}(e, (a, b, c)) &= (e, (a, b))
\end{align*}
Contracting away $e$ and the first element of the pair, we get the three equations
\begin{align*}
d_{0}(ba\inv, ca\inv) &= cb\inv\\
d_{1}(ba\inv, ca\inv) &= ca\inv\\
d_{2}(ba\inv, ca\inv) &= ba\inv
\end{align*}
and indeed, we have
$d_{1}(ba\inv, ca\inv) = d_{0}(ba\inv, ca\inv)d_{2}(ba\inv, ca\inv).$
This is equivalent, but not quite the same, as the standard presentation. It amounts to
\begin{align*}
  d_{0}(g, h) &= hg\inv \\
  d_{1}(g, h) &= h \\
  d_{2}(g, h) &= g.
\end{align*}
\end{proof}

Using the \v{C}ech nerve, we can extract all the coherence conditions governing a homomorphism of higher groups. We first note that the realization of the \v{C}ech nerve of a group is a delooping of it.
\begin{prop}
Let $G$ be a $0$-skeletal higher group with simplicially crisp delooping $\B G$, and let $\check{\type{C}}(G)$ be the \v{C}ech nerve of the basepoint inclusion $\pt_{\B G} : \ast \to \B G$. Then the projection $\fst : \check{\type{C}}(G) \to \B G$ is a $\geomreal$-unit.
\end{prop}
\begin{proof}
By Theorem 5.9 of \cite{jaz:good-fibrations}, $\B G$ is $0$-skeletal. By Axiom
$\Delta\sk$, it is therefore $\geomreal$-modal. Therefore, to show that $\fst :
\check{\type{C}}(G) \to \B G$ is a $\geomreal$-unit, it suffices to show that it
is $\geomreal$-connected. Since $\B G$ is $0$-connected, it suffices to show
that the fiber over the base point is $\geomreal$-connected, and this fiber is
equivalent to $\csk G$. This follows by \cref{thm:re.csk.trivial}; $\geomreal \csk G$ is
contractible.
\end{proof}

\begin{prop}
  Let $G$ and $H$ be $0$-skeletal higher groups. Then the type of homomorphisms $G \to H$ is equivalent to the type of pointed maps $\check{\type{C}}(G) \pto \check{\type{C}}(H)$.
\end{prop}
\begin{proof}

  Recall that a homomorphism of higher groups is by definition a pointed map
between their deloopings. That is, a homomorphism $\varphi : G \to H$ is
equivalently a diagram as on the left, while a pointed map between the \v{C}ech
nerves is a diagram as on the right:
\[
  \left\{  \begin{tikzcd}
    \ast & \ast \\
    {\B G} & {\B H}
    \arrow["{\pt_{\B G}}"', from=1-1, to=2-1]
    \arrow["{\B \varphi}"', dashed, from=2-1, to=2-2]
    \arrow[from=1-1, to=1-2, equals]
    \arrow["{\pt_{\B H}}", from=1-2, to=2-2]
  \end{tikzcd}
\right\}
  \quad\overset{?}{\equiv}\quad
 \left\{   \begin{tikzcd}
    \ast & \ast \\
    {\check{\type{C}}(G)} & {\check{\type{C}}(H)}
    \arrow["{\pt_{\B G}}"', from=1-1, to=2-1]
    \arrow["f"', dashed, from=2-1, to=2-2]
    \arrow[from=1-1, to=1-2, equals]
    \arrow["{\pt_{\B H}}", from=1-2, to=2-2]
  \end{tikzcd}
\right\}
\]
We are aiming for an equivalence between these two types, which we may present
as a one-to-one correspondence. So, to $\B \varphi : \B G \to \B H$
and $\pt_{\B G} : \pt_{\B H} = \B \varphi (\pt_{\B G})$ and $f :
\check{\type{C}}(G) \to \check{\type{C}}(H)$ and $\pt_f :
\pt_{\check{\type{C}}(H)} = f(\pt_{\check{\type{C}}(G)})$ associate the type
\[
  \dsum{\square : \dprod{x : \check{\type{C}}(G)} (\B \varphi(\fst x) = \fst (f
    x))} (\pt_{\B\varphi} \cdot \square(\pt_{\check{\type{C}}(G)}) = \fst_\ast \pt_f)
\]
which, diagrammatically, is the type of witnesses that the following diagram
commutes:
\[
  \begin{tikzcd}
    & \ast & \ast \\
    {\check{\type{C}}(G)} & {\check{\type{C}}(H)} \\
    && {\B G} & {\B H}
    \arrow[from=1-2, to=2-1]
    \arrow[from=1-3, to=2-2, shorten >= -8pt]
    \arrow["f", from=2-1, to=2-2]
    \arrow[from=1-2, to=1-3, equals]
    \arrow[from=1-3, to=3-4]
    \arrow["{\B \varphi}"', from=3-3, to=3-4]
    \arrow["{\fst}",from=2-2, to=3-4]
    \arrow["{\fst}",from=2-1, to=3-3]
    \arrow[from=1-2, to=3-3, crossing over]
  \end{tikzcd}
\]
To show that this gives a one-to-one correspondence means showing that the types
of diagrams
\[
  \left\{
      \begin{tikzcd}
        & \ast & \ast \\
        {\check{\type{C}}(G)} & {\check{\type{C}}(H)} \\
        && {\B G} & {\B H}
        \arrow[from=1-2, to=2-1]
        \arrow[from=1-3, to=2-2, shorten >= -8pt]
        \arrow["f", dashed, from=2-1, to=2-2]
        \arrow[from=1-2, to=1-3]
        \arrow[from=1-3, to=3-4]
        \arrow["{\B \varphi}"', from=3-3, to=3-4]
        \arrow[from=2-2, to=3-4]
        \arrow[from=2-1, to=3-3]
        \arrow[from=1-2, to=3-3, crossing over]
      \end{tikzcd}
    \right\}
  \quad\mbox{and}\quad
  \left\{
      \begin{tikzcd}
        & \ast & \ast \\
        {\check{\type{C}}(G)} & {\check{\type{C}}(H)} \\
        && {\B G} & {\B H}
        \arrow[from=1-2, to=2-1]
        \arrow[from=1-3, to=2-2, shorten >= -8pt]
        \arrow["f",  from=2-1, to=2-2]
        \arrow[from=1-2, to=1-3]
        \arrow[from=1-3, to=3-4]
        \arrow["{\B \varphi}"',dashed, from=3-3, to=3-4]
        \arrow[from=2-2, to=3-4]
        \arrow[from=2-1, to=3-3]
        \arrow[from=1-2, to=3-3, crossing over]
      \end{tikzcd}
    \right\}
\]
are both contractible, the left for any homomorphism $\varphi$ and the right for
any pointed map $f$.

Let $\varphi$ be a homomorphism. Since by definition $\check{\type{C}}(G)$ and $\check{\type{C}}(H)$ were the
$\csk$-factorizations of the basepoint inclusions, there is a unique
filler of this square:
\[
  \begin{tikzcd}
    \ast & \ast & {\check{\type{C}}(H)} \\
    {\check{\type{C}}(G)} & {\B G} & {\B H}
    \arrow["{\B \varphi}"', from=2-2, to=2-3]
    \arrow[from=1-1, to=1-2, equals]
    \arrow["{\pt_{\B G}}"', from=1-1, to=2-1]
    \arrow["{\pt_{\B H}}", from=1-2, to=1-3]
    \arrow[from=2-1, to=2-2]
    \arrow[from=1-3, to=2-3]
    \arrow["{\exists!}", dashed, from=2-1, to=1-3]
  \end{tikzcd}
\]
But this is precisely a rearrangement of the diagram on the left.

Similarly, if $f$ is a pointed map, then $\geomreal f : \B G \to \B
H$ makes the diagram on the right commute, and by the universal property of the
$\geomreal$-unit this is the unique such map.
\end{proof}

\subsection{Global Equivariant Cohesion}

In \emph{Global Homotopy Theory and Cohesion} \cite{rezk:global-cohesion}, Rezk shows that the
$\infty$-topos of \emph{global equivariant homotopy types} is cohesive over the
$\infty$-topos of homotopy types. While Rezk constructs his site out of all
compact Lie groups, we will follow Sati and Schreiber
\cite{sati-schreiber:proper-orbifold-cohomology} in restricting our attention to
the finite groups. The global orbit category $\type{Glo}$ is defined to be the
full subcategory of homotopy types spanned by the deloopings $\B G$ of finite
groups $G$. This is a $(2, 1)$-category, and the global equivariant topos is
defined to be the $\infty$-category of homotopy type valued presheaves on it.

There is an adjoint quadruple connecting the global equivariant topos and the
topos of homotopy types:
\[
\begin{tikzcd}
  \Ho^{\type{Glo}\op} \\
  \Ho \ar[u, leftarrow,shift left = 15, "\colim" ] \ar[u,  shift left =
  5, "\Delta"] \ar[u, shift
  right = 5, leftarrow, "\Gamma"] \ar[u, shift right = 15, "\nabla"]
\end{tikzcd}
\]
\begin{itemize}
\item $\colim X$ is the colimit of the functor $X : \type{Glo}\op \to \Ho$,
  which takes the \emph{strict quotient} of the global equivariant homotopy type $X$.
\item $\Delta S$ is the inclusion of constant functors: $\Delta  X (\B G) :\defeq
  S$. We will refer to such equivariant types as \emph{invariant} types.
\item $\Gamma X :\defeq X(\ast)$ is the evaluation at the point. This is known
  as the \emph{homotopy quotient} of the global equivariant homotopy type $X$.
\item $\nabla S$ is the Yoneda embedding: $\nabla S(\B G) :\defeq \Ho(\B G, S)$.
\end{itemize}

This adjoint quadruple gives rise to the cohesive modalities $\orbishape \dashv \orbiflat \dashv \orbisharp$ of equivariant
cohesion:
\begin{itemize}
\item The (``equivariant shape'') \emph{strict quotient} modality $X \mapsto \orbishape X$ sends a
  global equivariant type to its strict quotient, considered as an invariant type.
\item The (``equivariant flat'') \emph{homotopy quotient} modality $X \mapsto \orbiflat X$ sends a
  global equivariant type to its homotopy quotient, considered as an invariant
  type. Internally speaking, we say that an equivariantly crisp type is
  invariant when it is $\orbiflat$-modal.
\item The (``equivariant sharp'') \emph{orbisingular} modality $X \mapsto \orbisharp X$ sends a global
  equivariant type to its homotopy quotient, but considered with its natural
  equivariance via maps from the deloopings of finite groups.
\end{itemize}

Our axioms for global equivariant cohesion are quite straightforward:
\begin{axiom}[Global Equivariant Axioms]
  The type family $\orbisharp \B : \type{FinGrp} \to \Type$ sending a finite
  group $G$ to $\orbisharp \B G$ detects equivariant continuity and connectivity.
\end{axiom}

The types $\orbisharp \B G$ for finite groups $G$ are the
\emph{orbi-singularities}. By the definition above, we may recover $X(\B G)$
(considered with its natural equivariance) as
$\orbisharp (\orbisharp \B G \to X)$.

\begin{rmk}
The family $\orbisharp \B G$ for finite
groups $G$ is a large family, but we may reduce it to a small family by noting
that the type of finite groups is essentially small. This is a useful
observation, since it allows us to conclude that $\orbishape$, defined by
nullifying all $\orbisharp \B G$, is an accessible modality.
\end{rmk}

\begin{rmk}
Global equivariant cohesion shares a feature with Shulman's continuous real
cohesion: both are \emph{definable} in the sense that the types which detect
continuity and connectivity are definable without axioms in the type theory. This
is not the case for simplicial cohesion, which appears to require postulating
the $1$-simplex. It is not clear to us whether there are any general features
shared by definable cohesions.
\end{rmk}

In \emph{Proper Orbifold Cohomology} \cite{sati-schreiber:proper-orbifold-cohomology}, Sati and Schreiber work with
equivariant differential cohesion to give an abstract account of the
differential cohomology of orbifolds. We can prove some of their lemmas easily
in global equivariant cohesion; we will return to prove the lemmas relating
equivariant and differential cohesion in the upcoming
\cref{sec:equivariant.differential}.

The following lemma appears as Proposition 3.62 in \cite{sati-schreiber:proper-orbifold-cohomology}.
\begin{lem}
We have the following equivalences for the generic orbi-singularities
$\orbisharp \B G$:
\[
\orbishape \orbisharp \B G \equiv \ast \quad\quad \orbiflat \orbisharp \B G
\equiv \B G \quad\quad \orbisharp \orbisharp \B G \equiv \orbisharp \B G.
\]
\end{lem}
\begin{proof}
The first equivalence follows by the assumption that $\orbisharp \B G$ detects
equivariant connectivity. The second follows by combining Theorem 6.22 of
\cite{mike:real-cohesive-hott} to see that $\orbiflat \orbishape \B G \equiv
\orbiflat \B G$ with Theorem 5.9 of \cite{jaz:good-fibrations} to note that
since $G$ is crisply $\orbiflat$-modal (as a finite set), $\B G$ is as well. The
third is simply the idempotence of $\orbisharp$.
\end{proof}

The following lemma is a slight strengthening of Lemma 3.65 of
\cite{sati-schreiber:proper-orbifold-cohomology}.
\begin{lem}\label{lem:equivariant.sets.invariant}
  Let $X$ be an equivariantly crisp set. Then $X$ is both invariant
  ($\orbiflat$-modal) and orbi-singular
  ($\orbisharp$-modal).
\end{lem}
\begin{proof}
In both cases we will use that $\orbisharp \B G$ detects equivariant continuity.
To show that $X$ is invariant, we must show that the $\orbiflat$-counit is an
equivalence. By \cref{thm:equivalence.on.generators.gives.equivalence}, it
suffices to show that the map
\[
\orbiflat(\orbisharp \B G \to \orbiflat X) \to \orbiflat(\orbisharp \B G \to X)
\]
is an equivalence. But $\orbisharp$ is a lex modality and $\B G$ is
$0$-connected; therefore, $\orbisharp \B G$ is $0$-connected. Furthermore,
$\orbiflat X$ is a set since $X$ is, using Corollary 6.7 of
\cite{mike:real-cohesive-hott}. Therefore, the above map is equivalent to
$\orbiflat \varepsilon : \orbiflat \orbiflat X \to \orbiflat X$,
which is an equivalence.

Similarly, to show that $X$ is orbi-singular, it suffices to show
that the map
\[
\orbiflat(\orbisharp \B G \to X) \to \orbiflat ( \B G \to X)
\]
given by pre-composing with the $\orbiflat$-counit of $\orbisharp \B G$ is an
equivalence. But again, by the connectivity of $\orbisharp \B G$ and $
\B G$, this map is equivalent to the identity $\orbiflat X \to \orbiflat X$,
which is an equivalence.
\end{proof}

\begin{rmk}
Miller's theorem (formerly the
Sullivan conjecture) states that
the space of maps $\B G \to X$ with
$G$ a finite group and $X$ a finite
cell complex is equivalent to $X$.
In equivariant modal terms, this
says that finite cell complexes (the
closure of the class $\{\emptyset,\,
\ast\}$ under pushout) are
$\orbisharp$-modal. It is not likely
that this theorem could be proven on
purely modal grounds. However, if
Miller's theorem
were proven in ordinary HoTT, then
the modal statement could be proven in a manner
similar to
\cref{lem:equivariant.sets.invariant}
but instead of appealing to the
truncatedness of $X$, appealing to
the proof of Miller's theorem (since
any crisp finite cell complex is
$\orbiflat$-modal as a crisp pushout
of $\orbiflat$ types).
\end{rmk}

\subsection{Topological Toposes}

Johnstone defined his \emph{topological topos} in
\cite{johnstone:topological-topos} in order to provide a topos of spaces for
which the geometric realization of simplicial sets was a geometric morphism. The
problem with using a real-cohesive topos for this purpose (as suggested
previously by Lawvere) is the failure of the \emph{analytic lesser limited
  principle of omniscience} which says that $\Rb_D =
(-\infty, 0] \cup [0, \infty)$ (see Theorem 11.7 of \cite{mike:real-cohesive-hott}
and the discussion in Section 8.3 of \emph{ibid.}). This failure means that
gluing together simplices along their (closed) faces gives the wrong topology on
the resulting space.

Johnstone remedies this by changing the test space from the real numbers to the
walking convergent sequence $\Nb_{\infty}$. The walking convergent sequence may
be defined internally as the set of monotone functions $\Nb \to \{0, 1\}$ (see
for example
\cite{escardo:extended-nats}). Therefore, it is rather
straightforward to give an internal axiomatization for Johnstone's topos.

\begin{axiom}[Topological Focus]
The topological focus is determined by asserting that $\Nb_{\infty}$ detects
topological continuity.
\end{axiom}

We may also be able to determine
condensed homotopy types as in
\cite{clausen-sholze:condensed} ---
or rather the similar but more
topos-theoretic pyknotic homotopy
types of
\cite{barwick-haine:pyknotic} ---
using a similar axiom. Define a
profinite set to be the limit of a
crisp diagram of finite sets indexed by a discrete ($\flat$-modal)
partially ordered set with decidable
order. We may then assert that the
family of profinite sets detects
condensed continuity.

However, we do not know if these
axioms are sufficient for proving
theorems in these topological toposes.

\section{Multiple Focuses}\label{sec:multiple.focuses}

Now, we turn our attention to generalities on possible relationships between different
focuses. For this section, fix two focuses $\focusA$ and $\focusB$. First, we
should show that the focuses do indeed commute:
\begin{prop}\label{prop:sharps.commute}
  For any type $B$, the map
  $\sharp_{\focusA} \sharp_{\focusB} B \to \sharp_{\focusB} \sharp_{\focusA} B$ defined by
  \begin{align*}
    x \mapsto x_{\sharp_{\focusA}}{}_{\sharp_{\focusB}}{}^{\sharp_{\focusA}}{}^{\sharp_{\focusB}}
  \end{align*}
  is an equivalence. Furthermore, the maps $\sharp_{\focusA}\sharp_{\focusB} B
  \to \sharp_{\focusA\focusB} B$ and $\sharp_{\focusB}\sharp_{\focusA}B \to
  \sharp_{\focusA\focusB} B$ defined by
  \begin{align*}
   x &\mapsto x_{\sharp_{\focusA}\sharp_{\focusB}}{}^{\sharp_{\focusA\focusB}} \\
   x &\mapsto x_{\sharp_{\focusB}\sharp_{\focusA}}{}^{\sharp_{\focusA\focusB}}
  \end{align*}
  are also equivalences.
\end{prop}
\begin{proof}
  The first map is well-defined because the use of $\sharp_{\focusA}$- and
  $\sharp_{\focusB}$-introduction means that the assumption $x$ becomes crisp for both
  $\focusA$ and $\focusB$, so we may apply $\sharp_{\focusA}$- and then
  $\sharp_{\focusB}$-elimination to it. We may similarly define an inverse by
  \[
x \mapsto x_{\sharp_{\focusB}\sharp_{\focusA}}{}^{\sharp_{\focusB}\sharp_{\focusA}}.
  \]
  These maps are definitional inverses by the computation rules for
  $\sharp$s.

  The other maps are similarly well defined since being crisp for both $\focusA$
  and $\focusB$ means being crisp for focus $\focusA\focusB$. The inverse may be
  defined in the straightforward way.
\end{proof}

We also note that the ordering of focuses is reflected in the containment of
their $\sharp$-modal (and so also $\flat$-modal) types.
\begin{cor}\label{lem:orderd.sharps}
  Suppose that $\focusB \leq \focusA$. Then any $\sharp_{\focusB}$-modal type is
  $\sharp_{\focusA}$-modal.
\end{cor}
\begin{proof}
  Since $\focusB \leq \focusA$ is defined to mean
  $\focusA\focusB \defeq \focusB$, we know that
  $\sharp_{\focusA\focusB} A \defeq \sharp_{\focusB} A$.
  By assumption $\sharp_{\focusB} A \equiv A$, and chaining this with
  the commutativity equivalence of \cref{prop:sharps.commute}
  \[
\sharp_{\focusA} A \simeq \sharp_{\focusA} \sharp_{\focusB} A \simeq
\sharp_{\focusA\focusB} A \defeq \sharp_{\focusB} A \simeq A.
  \]
  Tracing these simple equivalences through, this does indeed give an inverse to
  the $\sharp_{\focusA}$-unit.
\end{proof}

We can similarly show that the $\flat$s commute. This is made simpler through
the use of crisp induction.

\begin{prop}\label{prop:flats-commute}
  Let $A$ be an $\focusA\focusB$-crisp type. Then
  $\flat_{\focusA} \flat_{\focusB} A \to \flat_{\focusB} \flat_{\focusA} A$ defined by
  \begin{align*}
    u \mapsto \ind{u}{v^{\flat_{\focusA}}}{(\ind{v}{w^{\flat_{\focusB}}}{(w^{\flat_{\focusA}})^{\flat_{\focusB}}})}
  \end{align*}
  is an equivalence, natural in $A$.
\end{prop}
\begin{proof}
  In words, the map is defined as follows.  Performing $\flat_{\focusA}$-induction on
  $u : \flat_{\focusA} \flat_{\focusB} A$ gives an assumption $v :_{\focusA} \flat_{\focusB}
  A$. A second induction on the term $v : \flat_{\focusB} A$ then gives us an
  assumption $w :_{\focusA\focusB} A$. This second induction is
  `$\focusA$-crisp $\flat_{\focusB}$-induction': the resulting assumption $w$
  inherits the $\focusA$-crispness of term $v$ and gains $\focusB$-crispness
  from the removal of $\flat_{\focusB}$. Finally, we form
  $(w^{\flat_{\focusA}})^{\flat_{\focusB}}$ by applying $\flat$-introduction twice. A
  map in the other direction is constructed in the same way, and then the proofs
  these are inverse are immediate by another pair of inductions each.
\end{proof}

We note also that the $\flat$-inductions commute.
\begin{lem}
  $\flat_{\focusA}$-induction and $\flat_{\focusB}$-induction commute:
  \begin{align*}
    \ind{(\ind{M}{v^{\flat_{\focusA}}}{N})}{u^{\flat_{\focusB}}}{C} = \ind{M}{v^{\flat_{\focusA}}}{(\ind{N}{u^{\flat_{\focusB}}}{C})}
  \end{align*}
  (when this is well-typed, i.e., $v$ does not occur in $C$.)
\end{lem}
\begin{proof}
  First using uniqueness of $\flat_{\focusB}$, we have
  \begin{align*}
    &\ind{(\ind{M}{v^{\flat_{\focusA}}}{N})}{u^{\flat_{\focusB}}}{C} \\
    &= \ind{M}{w^{\flat_{\focusA}}}{(\ind{(\ind{w^{\flat_{\focusA}}}{v^{\flat_{\focusA}}}{N})}{u^{\flat_{\focusB}}}{C})} \\
    &\defeq \ind{M}{w^{\flat_{\focusA}}}{(\ind{N[w/v]}{u^{\flat_{\focusB}}}{C})} \\
    &\defeq \ind{M}{v^{\flat_{\focusA}}}{(\ind{N}{u^{\flat_{\focusB}}}{C})}
  \end{align*}
\end{proof}

\subsection{Commuting Cohesions}

Now let's turn our attention to the relationships between two commuting
cohesions.
\begin{lem}
  Suppose that $\focusA$ and $\focusB$ are both cohesive. If a $\focusA\focusB$-crisp type $A$ is $\flat_{\focusA}$-modal, then
  $\shape_{\focusB} A$ is still $\flat_{\focusA}$-modal.
\end{lem}
\begin{proof}
  We need to produce an inverse to the counit
  $\varepsilon_{\focusA} : \flat_{\focusA} \shape_{\focusB} A \to \shape_{\focusB} A$.
  First construct the composite
  \begin{align*}
    A \xto{s} \flat_{\focusA} A \xto{\flat_{\focusA} \eta_{\focusB}} \flat_{\focusA} \shape_{\focusB} A
  \end{align*}
  where $A \xto{s} \flat_{\focusA} A$ is the assumed inverse to $\varepsilon_{\focusA} : \flat_{\focusA} A \to A$.

  The type $\flat_{\focusA} \shape_{\focusB} A$ is certainly $\flat_{\focusB}$-modal by
  commutativity of $\flat_{\focusA}$ and $\flat_{\focusB}$:
  \begin{align*}
    \flat_{\focusB} \flat_{\focusA} \shape_{\focusB} A \equiv \flat_{\focusA} \flat_{\focusB} \shape_{\focusB} A
    \equiv \flat_{\focusA} \shape_{\focusB} A
  \end{align*}
  and therefore the above map factors through
  $i : \shape_{\focusB} A \to \flat_{\focusA} \shape_{\focusB} A$, our purported inverse.

  For one direction, consider the naturality square for the counit $\varepsilon_{\focusA}$:
  \[
    \begin{tikzcd}
      \flat_{\focusA} \shape_{\focusB} A \ar[d, "\varepsilon_{\focusA}" swap] \ar[r, "\flat_{\focusA} i"] & \flat_{\focusA} \flat_{\focusA}
      \shape_{\focusB} A  \ar[d, "\varepsilon_{\focusA}"] \\
      \shape_{\focusB} A \ar[r, "i" swap] & \flat_{\focusA} \shape_{\focusB} A
    \end{tikzcd}
  \]
  The map $\flat_{\focusA} i$ is equal to the comultiplication
  $\delta_{\focusA} : \flat_{\focusA} A \to \flat_{\focusA} \flat_{\focusA} A$
  defined by $a^{\flat_{\focusA}} \mapsto a^{\flat_{\focusA}\flat_{\focusA}}$, because both are
  inverse to the map
  $\flat_{\focusA} \varepsilon_{\focusA} : \flat_{\focusA} \flat_{\focusA} A \to \flat_{\focusA}
  A$. And so the bottom composite in the square is equal to
  $\varepsilon_{\focusA} \circ \delta_{\focusA}$, which is the identity.

  In the other direction, because $\shape_{\focusB} A$ is $\shape_{\focusB}$-modal, it
  suffices to show that the composite
  \[A \to \shape_{\focusB} A \to \flat_{\focusA} \shape_{\focusB} A \to \shape_{\focusB} A\] is equal to
  the unit $\eta_{\focusB} : A \to \shape_{\focusB} A$. For this we have the
  following commutative diagram:
  \[
    \begin{tikzcd}
      A \ar[r, "\sim"] \ar[d, "{\eta_{\focusB}}" swap] & \flat_{\focusA} A \ar[d, "{\flat_{\focusA}\eta_{\focusB}}"] \ar[r, "{\varepsilon_{\focusA}}"] & A \ar[d, "{\eta_{\focusB}}"] \\
      \shape_{\focusB} A \ar[r, "i" swap] & \flat_{\focusA} \shape_{\focusB} A \ar[r, "{\varepsilon_{\focusA}}" swap] & \shape_{\focusB} A
    \end{tikzcd}
  \]
  The left square commutes by the definition of $i$, the right square by
  naturality of $\varepsilon_{\focusA}$, and the composite along the top is the identity.
\end{proof}

\begin{lem}\label{lem:shapes.commute}
  Suppose that $\focusA$ and $\focusB$ are both cohesive. Then
  $\shape_{\focusA} \shape_{\focusB} A \to \shape_{\focusB} \shape_{\focusA} A$
  is an equivalence for any $\focusA\focusB$-crisp type $A$.
\end{lem}
\begin{proof}
  The map
  $\eta_{\focusB} \circ \eta_{\focusA} : A \to \shape_{\focusA} A \to
  \shape_{\focusB} \shape_{\focusA} A$ factors through
  $\shape_{\focusA} \shape_{\focusB} A$, because
  $\shape_{\focusB} \shape_{\focusA} A$ is $\flat_{\focusB}$-modal, as a type of
  the form $\flat_{\focusB} X$, and also $\flat_{\focusA}$-modal, by the
  previous lemma. The map the other way is defined similarly. To show these are
  inverses it suffices to show that they become so after precomposition with the
  composites of the units, because $\shape_{\focusB} \shape_{\focusA} A$ and
  $\shape_{\focusA} \shape_{\focusB} A$ are $\focusA\focusB$-discrete; this is
  immediate by the definition of the maps.
\end{proof}

We might also hope that, say, $\flat_{\focusA}$ and $\shape_{\focusB}$ commute
in general, but there is a useful sanity check that shows this is not
possible. In the bare type theory with no axioms, there is nothing that prevents
interpretation in a model where $\flat_{\focusA} \defeq \flat_{\focusB}$ and
$\shape_{\focusA} \defeq \shape_{\focusB}$. In ordinary cohesive type theory it
is certainly not the case that $\flat$ and $\shape$ commute, and so
$\flat_{\focusA} \shape_{\focusB} \equiv \shape_{\focusB} \flat_{\focusA}$
cannot be provable without further assumptions on $\focusA$ and $\focusB$.

A sufficient assumption on our focuses to make $\flat_{\focusA}$ and $\shape_{\focusB}$ commute in
this way is the following:
\begin{defn}\label{defn:orthogonal.cohesions}
  Suppose that $\Afam :_{\focusA\focusB} I \to \Type$ and $\Bfam :_{\focusA\focusB} J \to
  \Type$ detect $\focusA$ and $\focusB$ connectivity respectively. We say that
  focuses $\focusA$ and $\focusB$ are \emph{orthogonal} if $\Afam_i$ is
  $\flat_{\focusB}$-modal for all $i$, and $\Bfam_{\! j}$ is $\flat_{\focusA}$-modal
  for all $j$.
\end{defn}

Our present goal is to show that this indeed makes
$\flat_{\focusA} \shape_{\focusB} \equiv \shape_{\focusB} \flat_{\focusA}$.
We will in fact only use that the $\Afam_i$ are
$\flat_{\focusB}$-modal; of course the dual results, flipping $\focusA$ and
$\focusB$, require the other half of orthogonality.

\begin{lem}\label{lem:gen-flat-preserves-discrete}
 Let $\focusA$ and $\focusB$ be cohesive focuses that are orthogonal. Then for any $\focusB$-crisp $A$, if $A$
is $\shape_{\focusA}$-modal, $\flat_{\focusB} A$ is still $\shape_{\focusA}$-modal.
\end{lem}
\begin{proof}
  Our goal is to show that $\flat_{\focusB} A$ is equivalent to $\Afam_i \to \flat_{\focusB} A$
  via precomposition by $\Afam_i \to 1$, for any $i : I$. We easily check that the
  type $\Afam_i \to \flat_{\focusB} A$ is $\focusB$-discrete: for any $\Bfam_{\! j}$,
  \begin{align*}
    \Bfam_{\! j} \to (\Afam_i \to \flat_{\focusB} A)
    &\equiv \Bfam_{\! j} \to (\Afam_i \to \flat_{\focusB} A) \\
    &\equiv \Afam_i \to (\Bfam_{\! j} \to \flat_{\focusB} A) \\
    &\equiv \Afam_i \to \flat_{\focusB} A
  \end{align*}
  because $\flat_{\focusB} A$ is $\focusB$-discrete. Then, by adjointness of
  $\shape_{\focusB}$ and $\flat_{\focusB}$:
  \begin{align*}
    (\Afam_i \to \flat_{\focusB} A)
    &\equiv \flat_{\focusB} (\Afam_i \to \flat_{\focusB} A) \\
    &\equiv \flat_{\focusB} (\shape_{\focusB} \Afam_i \to A) \\
    &\equiv \flat_{\focusB} (\Afam_i \to A) \\
    &\equiv \flat_{\focusB} A
  \end{align*}
\end{proof}

\begin{prop}[Crisp $\shape_{\focusA}$-induction]
  Suppose that $\focusA$ and $\focusB$ are cohesive and orthogonal. If $B$ is
  $\focusA$-discrete and $\focusB$-crisp, then for any $\focusB$-crisp $A$ the map
  \begin{align*}
    \flat_{\focusB}(\flat_{\focusB} \shape_{\focusA} A \to B) \to \flat_{\focusB} (\flat_{\focusB} A \to B)
  \end{align*}
  given by precomposition by
  $\flat_{\focusB} \eta_{\focusA} : \flat_{\focusB} A \to \flat_{\focusB} \shape_{\focusA} A$ is an
  equivalence.
\end{prop}

\begin{rmk}
As a rule, crisp $\shape$-induction would be written:
\begin{mathpar}
  \inferrule{\focusB \setminus \Gamma, x :_{\focusB} {\shape_{\focusA}} A \yields C \type \and
    \focusB \setminus \Gamma, x :_{\focusB} {\shape_{\focusA}} A \yields w : \mbox{is-$\flat_{\focusA}$-modal}(C) \\
    \focusB \setminus \Gamma \yields M : {\shape_{\focusA}} A \\
    \focusB \setminus \Gamma, u :_{\focusB} A \yields N : C[u^{\shape_{\focusA}}/x]}
  {\Gamma \yields (\ind{M}{u^{\shape_{\focusA}}}{N}):C[M/x]}
\end{mathpar}
\end{rmk}

\begin{proof}
  We deploy the usual trick for deriving crisp induction principles: using
  $\sharp_{\focusB}$ to move the $\flat_{\focusB}$ out of the way. Crucially, the previous
  proposition is what allows us to apply the universal property of $\shape_{\focusA}$
  on maps into $\sharp_{\focusB} B$.
  \begin{align*}
    \flat_{\focusB}(\flat_{\focusB} \shape_{\focusA} A \to B)
    &\equiv \flat_{\focusB} (\shape_{\focusA} A \to \sharp_{\focusB} B) \\
    &\equiv \flat_{\focusB} (A \to \sharp_{\focusB} B) \\
    &\equiv \flat_{\focusB} (\flat_{\focusB} A \to B)
  \end{align*}
\end{proof}

\begin{prop}\label{lem:shape.flat.commute}
  If $\focusA$ and $\focusB$ are orthogonal and cohesive focuses, then $\shape_{\focusA}$ and $\flat_{\focusB}$ commute on $\focusB$-crisp
  types.
\end{prop}
\begin{proof}
  This is now straightforward induction on $\shape_{\focusA}$ and $\flat_{\focusB}$ in both
  directions, using crisp $\shape_{\focusA}$-induction when defining
  $\flat_{\focusB} \shape_{\focusA} A \to \shape_{\focusA} \flat_{\focusB} A$ and
  \cref{lem:gen-flat-preserves-discrete} when defining
  $\shape_{\focusA} \flat_{\focusB} A \to \flat_{\focusB} \shape_{\focusA} A$ to know that
  $\flat_{\focusB} \shape_{\focusA} A$ is $\focusA$-discrete.
\end{proof}

\begin{cor}\label{lem:flat.sharp.commute}
If $\focusA$ and $\focusB$ are
orthogonal and cohesive focuses,
then $\flat_{\focusA}$ and
$\sharp_{\focusB}$ commute on
$\focusA\focusB$-crisp types.
\end{cor}
\begin{proof}
On $\focusA\focusB$-crisp types,
$\flat_{\focusA}\sharp_{\focusB}$ is
right adjoint to
$\flat_{\focusB}\shape_{\focusA}$,
and
$\sharp_{\focusB}\flat_{\focusA}$ is
right adjoint to
$\shape_{\focusA}\flat_{\focusB}$.
Therefore, if
$\flat_{\focusB}\shape_{\focusA} X
\simeq
\shape_{\focusA}\flat_{\focusB}$ for
a $\focusA\focusB$-crisp type $X$,
then also
$\flat_{\focusA}\sharp_{\focusB}X
\simeq \sharp_{\focusB}\flat_{\focusA}$.
\end{proof}

Next we investigate a relationship between $\shape$ and $\sharp$. Again, this
depends on a relationship between the families which detect continuity and
connectivity of the two focuses.

\begin{prop}\label{lem:shape.sharp.orthogonal}
  Suppose that $\focusB$ is cohesive, and that the following hold:
  \begin{enumerate}
  \item $\Bfam :_{\focusA\focusB} I \to \Type$ detects
    $\focusB$-connectivity, and $I$ is $\flat_{\focusA}$-modal.
  \item $\Bfam_i$ is $\flat_{\focusA}$-modal for all $i :_{\focusA} I$. (In
    particular, if $\focusA$ and $\focusB$ are orthogonal)
  \end{enumerate}
  Then if $X$ is $\shape_{\focusB}$-modal then $\sharp_{\focusA}X$ is also
  $\shape_{\focusB}$-modal.
\end{prop}
\begin{proof}
  Since $\Bfam$ detects $\focusB$-connectivity,
  it suffices to show that $\sharp_{\focusA} \shape_{\focusB} X$ is
  $\Bfam_i$-null for every $i : I$. Since $I$ is $\flat_{\focusA}$-modal, we may
  assume that $i :_{\focusA} I$ is $\focusA$-crisp by
  $\flat_{\focusA}$-induction. Then we can compute:
  \begin{align*}
    (\Bfam_i \to \sharp_{\focusA} X)
    &\equiv \sharp_{\focusA}(\Bfam_i \to \sharp_{\focusA} X)  \\
    &\equiv \sharp_{\focusA}(\flat_{\focusA} \Bfam_i \to  X) \\
    &\equiv \sharp_{\focusA}(\Bfam_i \to X) &\mbox{since $\Bfam_i$ was assumed $\flat_{\focusA}$-modal} \\
    &\equiv \sharp_{\focusA} X &\mbox{since $X$ is $\shape_{\focusB}$-modal}
  \end{align*}
Tracing upwards through this series of equivalences shows that the composite is
indeed the inclusion of constant functions.
\end{proof}

On the opposite extreme of orthogonality, we can see that if the $\Afam_i$ which
detect the connectivity of $\focusA$ are $\shape_{\focusB}$-connected, then any
$\shape_{\focusB}$-modal type is $\shape_{\focusA}$-modal.
\begin{prop}
Suppose that $\focusA$ and $\focusB$ are cohesive focuses where
$\Afam : I \to \Type$
detects the connectivity of $\focusA$.
Then the following are equivalent:
\begin{enumerate}
\item Every $\Afam_i$ is $\shape_{\focusB}$-connected.
\item
  Any $\shape_{\focusB}$-modal type is $\shape_{\focusA}$-modal.
\end{enumerate}
\end{prop}
\begin{proof}

Suppose that $\Afam_i$ is $\shape_{\focusB}$-connected for all $i$ and that $X$ is
$\shape_{\focusB}$-modal. We may compute:
\[
  (1 \to X) \equiv (\shape_{\focusB}\Afam_i \to X)
  \equiv (\Afam_i \to X)
\]
In the first equivalence, we use that $\Afam_i$ is $\shape_{\focusB}$-connected, and
in the second that $X$ is $\shape_{\focusB}$-modal. We conclude that $X$ is
$\shape_{\focusA}$-modal.

Conversely, suppose that any $\shape_{\focusB}$-modal type is
$\shape_{\focusA}$-modal. Then in particular $\shape_{\focusB} \Afam_i$ is
$\shape_{\focusA}$-modal, so that the identity map $\shape_{\focusB} \Afam_i \to
\shape_{\focusB} \Afam_i$ factors through $\shape_{\focusA} \shape_{\focusB} \Afam_i$.
But by \cref{lem:shapes.commute} we have
\[
\shape_{\focusA} \shape_{\focusB} \Afam_i \simeq \shape_{\focusB} \shape_{\focusA}
\Afam_i \simeq \shape_{\focusB} \ast \simeq \ast.
\]
Therefore, the identity of $\shape_{\focusB} \Afam_i$ factors through the point,
which means it is contractible.
\end{proof}

\section{Examples with Multiple Focuses}\label{sec:multiple.examples}

In this section, we will see examples with multiple focuses. In particular, we
will see simplicial real cohesion,
equivariant differential cohesion,
and supergeometric cohesion.

\subsection{Simplicial Real Cohesion}

We assume two basic focuses: the real (continuous or differential) focus
$\shape \dashv \flat \dashv \sharp$, and the simplicial focus
$\geomreal \dashv \sk \dashv \csk$. We will write $\Rb$ for whichever flavor of
real numbers is used in the real cohesive focus.

We will assume both the axioms of real cohesion and simplicial cohesion, as well
as the following axiom relating the two focuses.
\begin{axiom}[Simplicial Real Cohesion]
  We assume that the real focus and simplicial focus are orthogonal --- which is
  to say, $\Rb$ is $0$-skeletal and that $\Delta[1]$ is discrete. Furthermore,
  we assume that $\shape$ is computed pointwise: for any simplicially crisp
  type $X$, the action $(\eta_{\shape})_n : X_n \to (\shape X)_n$ of the
  $\shape$-unit of $X$ on $n$-simplices is itself a $\shape$-unit.
\end{axiom}

Our goal in this section will be to prove that if $M$ is a $0$-skeletal type ---
to be thought of as a ``manifold'', having only real-cohesive structure but no
simplicial structure --- and $U$ is a \emph{good cover} of $M$ --- one for which
the finite intersections are $\shape$-connected whenever they are inhabited ---
then the homotopy type $\shape M$ of $M$ may be constructed as the realization
of a discrete simplicial set --- namely, the \v{C}ech nerve of the open cover,
with each open replaced by the point.

\begin{defn}
  Let $M$ be a $0$-skeletal type. A \emph{cover} of $M$ consists of a
  discrete $0$-skeletal index set $I$, and a family $U : I \to (M \to \Prop)$ of
  subobjects of $M$ so that for every $m : M$ there is merely an $i : I$ with $m
  \in U_i$. We may assemble a cover into a single surjective map $c : \bigsqcup_{i
    : I} U_i \to M$, where
  \[
    \bigsqcup_{i : I} U_i :\defeq \dsum{i : I} \dsum{m : M} (m \in U_i).
  \]

  A cover $U : I \to (M \to \Prop)$ is a \emph{good cover} if for any $n
  : \Nb$ and any $k : [n] \to I$, the $\shape$-shape of the intersection
  \[
    \bigcap_{i : [n]} U_{k(i)} :\defeq \dsum{m : M} (\dprod{i : [n]} (m \in U_{k(i)})).
  \]
  is a proposition. That is, $\shape(U_{k(0)} \cap \cdots \cap U_{k(n)})$ is
  contractible whenever there is an element in the intersection.
\end{defn}

We begin with a few ground-setting lemmas.
\begin{lem}
  Let $U : I \to (M \to \Prop)$ be a simplicially crisp cover, and let $c : \bigsqcup_{i : I}
  U_i \to M$ be the associated covering map. Consider the projection
  $\pi : \check{\type{C}}(c) \to \csk I$ defined by $(m, z) \mapsto (\fst z_{\csk})^{
    \csk }$. Over a simplicially crisp $n$-simplex $k : \Delta[n] \to \csk I$, we have
  \[
\fib_{\pi_n}(k^{\sk}) \simeq \bigcap_{i : [n]} U_{k(i_{\sk})}.
  \]
  As a corollary, we have that
  \[
    \type{\check{C}}(c)_n \simeq \dsum{k : I^{[n]}} \bigcap_{i : [n]} U_{k(i)}.
  \]
\end{lem}
\begin{proof}
  We compute:
  \begin{align*}
    \fib_{\pi_n}(k^{\sk}) &\defeq \dsum{x : \check{\type{C}}(c)_n} (\pi_n x = k^{\sk}) \\
    &\equiv \dsum{(m, z) : \dsum{m : M}(\dsum{i : I}(m \in U_i))^{[n]}} (\pi_n e(m, z) = k^{\sk})\\
                          &\equiv \dsum{m : M}\dsum{K : [n] \to I} (p : \dprod{i : [n]} (m \in U_{K(i)}))\times (\pi_n e(m, i \mapsto (K(i), p)) = k^{\sk})\\
    \intertext{Here, $e(m,z)$ is image under the equivalence from \cref{prop:cech.nerve}. When all the modal dust settles, we will be left knowing that $\pi_ne(m, i \mapsto (K(i), p)) : \sk(\Delta[n] \to \csk I)$ is the unique correspondent to $K : [n] \to I$ under the $\sk \dashv \csk$ adjunction. Therefore, we may contract $K$ away with $k^{\sk}$ in the above type to get:}
    &\equiv \dsum{m : M}(\dprod{i : [n]} m \in U_{k(i_{\sk})}).\quad
  \end{align*}\qedhere
\end{proof}

For the next lemma, we will need to know that $\shape$ commutes with $\csk$ on suitably crisp types.
\begin{thm}\label{thm:shape.coskeleton.commute}
  For any simplicially crisp type $X$, we have that $\shape \csk X \simeq \csk
  \shape X$.
\end{thm}
\begin{proof}
  We know by \cref{lem:shape.sharp.orthogonal} that $\csk \shape X$ is $\shape$-modal. We therefore
 have a map $\shape \csk X \to \csk \shape X$ given as the unique factor of
 $\csk(-)^{\shape} : \csk X \to \csk \shape X$. We will show that this map is an
 equivalence. Since it is crisp, it suffices to show that it is an equivalence
 on $n$-simplices. To that end, we compute:
 \begin{align*}
   \sk(\Delta[n] \to \shape \csk X)
   &\equiv \shape \sk(\Delta[n] \to \csk X) \\
                                    &\equiv \shape \sk([n] \to X) \\
                                    &\equiv \sk(\shape([n] \to X)) \\
                                    &\equiv \sk ([n] \to \shape X) \\
   &\equiv \sk(\Delta[n] \to \csk \shape X)
 \end{align*}
 It remains to show that this is indeed the right equivalence. Since the first
 equivalence in the series above is given as the inverse of $(-)^{\shape}_n :
 (\csk X)_n \to (\shape \csk X)_n$, it suffices to check that given a crisp $z :
 \Delta[n] \to \csk X$, $(z^{\sk})^{\shape}$ corresponds under the above
 equivalences to $(\csk(-)^{\shape} \circ z)^{\sk}$. First, we send
 $(z^{\sk})^{\shape}$ to $((z \circ (-)_{\sk})^{\sk})^{\shape}$. Then, we send
 it to $((z \circ (-)_{\sk})^{\shape})^{\sk}$, and then to $(i \mapsto
 (z(i_{\sk})^{\shape}))^{\sk}$. Finally, we map this to $(i \mapsto
 ((z(i^{\sk}{}_{\sk})^{\shape}))^{\csk\sk}$, which does equal
 $(\csk(-)^{\shape} \circ z)(i) \defeq (z(i)^{\shape})^{\csk}$ at $i : \Delta[n]$.
\end{proof}

\begin{lem}
Let $U : I \to (M \to \Prop)$ be a simplicially crisp cover, and let $c : \bigsqcup_{i : I}
U_i \to M$ be the covering map itself. Consider the ``projection''
$\pi : \check{\type{C}}(c) \to \csk I$ defined by $(m, z) \mapsto (\fst z_{\csk})^{
  \csk }$. Then $U$ is a good cover if and only if the restriction $\pi :
\check{\type{C}}(c) \to \im \pi$ is a $\shape$-unit.
\end{lem}
\begin{proof}
  Since $\csk$ and $\shape$ commute by \cref{thm:shape.coskeleton.commute} and $I$ is discrete, $\csk I$ is also
  discrete. As the subtype of a discrete type, $\im \pi$ is discrete. Therefore,
  it suffices to show that $\pi : \check{\type{C}}(c) \to \im \pi$ induces an
  equivalence $\shape \check{\type{C}}(c) \xto{\sim} \im \pi$ if and only if the
  cover $U$ is good. Since $\pi$ is crisp, $\shape \check{\type{C}}(c) \to \im
  \pi$ is an equivalence if and only if it is an equivalence on all
  $n$-simplices. On $n$-simplices, this map (at the top of the following
  diagram) is equivalent to the map on the
  bottom of the following diagram:
  \[
    \begin{tikzcd}
      {(\shape\check{\type{C}}(c))_n} & {(\im \pi)_n} \\
      {\shape(\type{\check{C}}(c)_n)} & {\im \pi_n} \\
      {\shape\left(\dsum{k : I^{[n]}} \bigcap_{i : [n]} U_{k(i)}\right)} \\
      {\dsum{k : I^{[n]}} \shape\left(\bigcap_{i : [n]} U_{k(i)}\right)} & {\dsum{k : I^{[n]}} \exists \bigcap_{i : [n]}U_{k(i)}}
      \arrow[from=1-1, to=1-2]
      \arrow[from=2-1, to=2-2]
      \arrow[from=4-1, to=4-2]
      \arrow[from=1-1, to=2-1, equals]
      \arrow[from=2-1, to=3-1, equals]
      \arrow[from=3-1, to=4-1, equals]
      \arrow[from=1-2, to=2-2, equals]
      \arrow[from=2-2, to=4-2, equals]
    \end{tikzcd}
  \]
 The bottom map is an equivalence if and only if the cover is good, and so we
 conclude the same for the top map.
\end{proof}

Finally, we can piece these lemmas together for our result.

\begin{thm}\label{thm:good.cover.homotopy}
  Let $U : I \to (M \to \Prop)$ be a simplicially crisp good cover of a $0$-skeletal type $M$. Let
  $\pi : \type{\check{C}(U)} \to \csk I$ be the projection. Then
  \[
    \geomreal \im \pi \simeq \shape M.
  \]
  This exhibits the shape $\shape M$ as the realization of a discrete (simplicial) set.
\end{thm}
\begin{proof}
  Since the cover is good, we have that $\im \pi \simeq \shape \type{\check{C}}(U)$,
  so that
  \[
    \geomreal \im \pi \simeq \geomreal \shape \type{\check{C}}(U) \simeq \shape
    \geomreal \type{\check{C}}(U) \simeq \shape M.
  \]
  Now, since $I$ is a set and $\csk$ is a lex modality, $\csk I$ is also a set
  and so $\im \pi$ is a set as well. Furthermore, since subtypes of discrete types are
  discrete by Lemma 8.17 of \cite{mike:real-cohesive-hott}, and $\csk I$ is
  discrete since by \cref{thm:shape.coskeleton.commute} $\shape$ and $\csk$
  commute on simplicially crisp types, $\im \pi$ is discrete.
\end{proof}

\subsection{Equivariant Differential Cohesion}\label{sec:equivariant.differential}

In \emph{Proper Orbifold Cohomology}, Sati and Schreiber work in
equivariant differential cohesion to describe the differential cohomology of
orbifolds. This cohesion involves both the equivariant focus $\orbishape \dashv
\orbiflat \dashv \orbisharp$ and the differential (real-cohesive) focus $\shape
\dashv \flat \dashv \sharp$. In this section, we will assume both the axioms of
equivariant cohesion and differential real cohesion. We will refer to the smooth
reals by $\Rb$.

Unlike the simplicial real cohesive case, we do not need to add additional axioms to
ensure that the equivariant and differential cohesion are orthogonal.
\begin{lem}\label{lem:equivariant.differential.orthogonal}
  Equivariant and differential cohesion are orthogonal. That is:
  \begin{enumerate}
  \item The smooth reals $\Rb$ are invariant ($\orbiflat$-modal).
  \item For any finite group $G$, $\orbisharp \B G$ is discrete ($\flat$-modal).
  \end{enumerate}
\end{lem}
\begin{proof}
  Since the smooth reals are a set, they are invariant by
  \cref{lem:equivariant.sets.invariant}. Similarly, since $\B G$ is discrete and
  hence $\shape$-modal (by
  Theorem 5.9 of \cite{jaz:good-fibrations}), $\orbisharp \B G$ is still
  $\shape$-modal by \cref{lem:shape.sharp.orthogonal}.
\end{proof}

The following lemma appears as Lemma 3.67 of
\cite{sati-schreiber:proper-orbifold-cohomology},
and is proven quickly with our
general lemmas concerning orthogonal cohesions.
\begin{lem}
  Suppose that $X$ is both differentially and equivariantly crisp. Then
  \[
\orbiflat \shape X \equiv \shape \orbiflat X \quad\quad \orbiflat \flat X \equiv
\flat \orbiflat \flat X \quad\quad \orbiflat \sharp X \equiv \sharp \orbiflat X.
  \]
\end{lem}
\begin{proof}
  The first equivalence follows by \cref{lem:shape.flat.commute}. The second
  equivalence follows by
  \cref{prop:flats-commute}. The
  third follows by \cref{lem:flat.sharp.commute}.
\end{proof}

\subsection{Supergeometric Cohesion}

In his habilitation,
\emph{Differential Cohomology in a
  Cohesive $\infty$-Topos} \cite{urs:diff-coh},
Schreiber describes an increasing
tower of adjoint modalities which
appear in the setting of
supergeometry. The setting for
supergeometric cohesion --- called
``solid cohesion'' in \emph{ibid.}
--- is sheaves on the opposite of a
category of super
$\Ca^{\infty}$-algebras. Schreiber
calls these sheaves \emph{super formal
smooth $\infty$-groupoids}.
Specifically, the site is (the
opposite of) the full
subcategory of the category of super
commutative real algebras spanned by
objects of the form
\[
\Ca^{\infty}(\Rb^n) \otimes W
\otimes \Lambda \Rb^q
\]
where $W$ is a \emph{Weil algebra}
--- a commutative nilpotent
extension of $\Rb$ which is finitely
generated as an $\Rb$-module. The
factor $\Ca^{\infty}(\Rb^n) \otimes
W$ is even graded, while the
Grassmannian $\Lambda\Rb^q$ is odd
graded. See Definition 6.6.13 of \cite{urs:diff-coh}.

The inclusion of algebras of the
form $\Ca^{\infty}(\Rb^n) \otimes W$
has a left and a right adjoint. The
left adjoint is given by projecting
out the even subalgebra, and the
right adjoint is given by
quotienting by the ideal generated
by the odd graded elements. This
gives rise to an adjoint quadruple
between the resulting toposes of
sheaves and thus an adjoint triple of
idempotent adjoint (co)monads on the
topos of super formal smooth
$\infty$-groupoids:
${\rightrightarrows} \dashv
{\rightsquigarrow} \dashv {\supersharp}$.

Of these, $\rightrightarrows$ and
$\supersharp$ are idempotent monads.
However, $\rightrightarrows$ does not
preserve products, and so does not
give an internal modality. The
action of $\supersharp$ is easy
to define:
\[
\supersharp{X}
(\Ca^{\infty}(\Rb^n) \otimes W
\otimes \Lambda \Rb^q
) := X(\Ca^{\infty}(\Rb^n) \otimes W).
\]
That is, $\supersharp{X}$ is defined
by evaluating at the even part of
the superalgebra in the site. We may
characterize it internally by localizing
at the \emph{odd line} $\Rb^{0\mid
  1}$, which is the sheaf
represented by the free superalgebra
on one odd generator $\Lambda \Rb$.
We turn to the internal story now.

The topos of super formal smooth
$\infty$-groupoids also supports the
differential real cohesive
modalities $\shape \vdash \flat
\vdash \sharp$. These destroy all
geometric structure --- super and
otherwise. For this reason, we
will work with the lattice
$\{\mathrm{diff}< \mathrm{super} < \top\}$
of focuses.

The modalities of the
$\mathrm{super}$ focus are
${\rightsquigarrow} \vdash
\supersharp$. We will refer to
$\overset{\rightsquigarrow}{X}$ as
the even part of $X$, while
$\supersharp$ is known as the
\emph{rheonomic} modality.
We assume the following axioms for
supergeometric or \emph{solid} cohesion.
\begin{axiom}[Solid Cohesion]
Solid cohesion uses the focus
lattice $\{\mathrm{diff} <
\mathrm{super} < \top\}$. We use the
definition of real superalgebras due
to Carchedi and Roytenberg \cite{carchedi-roytenberg:superalgebra}.
\begin{enumerate}
\item We assume a commutative ring
  $\Rb^{1 \mid 0}$ satisfying the
  axioms of synthetic differential
  geometry (as e.g. in Section 4.1
  of \cite{jaz:orbifolds}) known as
  the \emph{smooth reals} or the
  \emph{even line}.
\item We assume an $\Rb^{1 \mid
    0}$-module $\Rb^{0 \mid 1}$.
  There is furthermore a bilinear
  multiplication $\Rb^{0 \mid 1}
  \times \Rb^{0 \mid 1} \to \Rb^{1
    \mid 0}$ which satisfies $a^2 =
  0$ for all $a : \Rb^{0 \mid 1}$.
  Together these axioms imply that
  $\Rb^{1 \mid 1} := \Rb^{1 \mid 0}
  \times \Rb^{0 \mid 1}$ is a
  $\Rb^{1 \mid 0}$-supercommutative
  superalgebra.
\item We assume the following odd
  form of the Kock-Lawvere axiom:
  For any function $f : \Rb^{0\mid
    1} \to \Rb^{0 \mid 1}$ with
  $f(0) = 0$, there is a unique $r :
  \Rb^{1 \mid 0}$ with $f(x) = rx$
  for all $x$.
  \item We assume that $\Rb^{0 \mid
      1}$ is $\rightsquigarrow$-connected.
\item We assume that $\Rb^{1\mid 0}$
  detects differential connectivity.
\item We assume that a type is
  $\supersharp$-modal if and only if
  it is $\Rb^{0 \mid 1}$-null.
\end{enumerate}
\end{axiom}

\begin{rmk}
It might seem prudent to instead
ask that differential connectivity
is detected by the family consisting
of both $\Rb^{1 \mid 0}$ and $\Rb^{0
\mid 1}$, since we want $\shape$ to
nullify all representables $\Rb^{n
  \mid q}$, but it suffices to test
with $\Rb^{1 \mid 0}$ since $\Rb^{0
  \mid 1}$ admits an explicit
contraction by its $\Rb^{1 \mid
  0}$-module structure (appealing to
Lemma 6.10 of \cite{jaz:good-fibrations}).
\end{rmk}

By \cref{lem:orderd.sharps}, any
$\sharp$-modal type is
$\supersharp$-modal. But also every
$\flat$-modal type is $\supersharp$-modal.
\begin{lem}
  If $X$ is $\shape$-modal (and, in
  particular, if $X$ is
  $\flat$-modal), then $X$ is $\supersharp$-modal.
\end{lem}
\begin{proof}
Since $\Rb^{0\mid 1}$ is
$\shape$-connected due to its
explicit contraction by the scaling
of its module structure, any
$\shape$-modal type is $\Rb^{0 \mid
  1}$-null, and therefore $\supersharp$-modal.
\end{proof}

\printbibliography

\appendix
\section{Proof Sketches for Admissible Rules}
\label{sec:proofs-admiss}

We sketch proofs that the operations on syntax are admissible, demonstrating the
interesting cases; those that involve division (\rulen{ctx-ext},
\rulen{$\flat$-form/intro/elim}, \rulen{$\sharp$-elim}) or promotion
(\rulen{$\sharp$-form/intro}).

\begin{defn}
  The $\focusA \Gamma$ and $\focusA \setminus \Gamma$ context operations extend
  in the obvious way to telescopes $\Gamma'$, so that
  \begin{align*}
    \focusA (\Gamma, \Gamma') &\defeq (\focusA \Gamma), (\focusA \Gamma') \\
    \focusA \setminus (\Gamma, \Gamma') &\defeq (\focusA \setminus \Gamma), (\focusA \setminus \Gamma')
  \end{align*}
\end{defn}

\begin{lem}[Weakening]
  Single-variable weakening is admissible.
  \begin{mathpar}
    \inferrule*[left=wk, fraction={-{\,-\,}-}]
    {\Gamma, \Gamma' \yields \judge}
    {\Gamma, w :_{\focusB} W, \Gamma' \yields \judge}
  \end{mathpar}
  Moreover, weakening does not change the size of the derivation tree.
\end{lem}
\begin{proof}
  Induction on $\judge$.
  \begin{description}[style=unboxed, labelwidth=\linewidth, font=\normalfont\itshape, listparindent=0pt]
  \item[Case (division).]
    \begin{mathpar}
      \inferrule*[left=$\flat$-form]
      {\focusA \setminus (\Gamma, w :_{\focusB} W, \Gamma') \yields A \jtype}
      {\Gamma, w :_{\focusB} W, \Gamma' \yields {\flat_{\focusA}} A \jtype}
    \end{mathpar}
    There are two subcases:
    \begin{itemize}
    \item If $\focusB \leq \focusA$: then $\focusA \setminus (\Gamma, w
      :_{\focusB} W, \Gamma') \defeq (\focusA \setminus \Gamma), w :_{\focusB}
      W, (\focusA \setminus \Gamma')$, in which case we induct on the type $A$
      and reapply the rule.
    \item If $\focusB \not\leq \focusA$: then
      $\focusA \setminus (\Gamma, w :_{\focusB} W, \Gamma') \defeq (\focusA
      \setminus \Gamma), (\focusA \setminus \Gamma') \defeq \focusA \setminus
      (\Gamma, \Gamma')$, and so already
      $\focusA \setminus (\Gamma, w :_{\focusB} W, \Gamma') \yields A \jtype$ and
      we can reapply the rule.
    \end{itemize}
  \item[Case (promotion).]
    \begin{mathpar}
      \inferrule*[left=$\sharp$-form]
      {\focusA(\Gamma, w :_{\focusB} W, \Gamma') \yields A \jtype}
      {\Gamma, w :_{\focusB} W, \Gamma' \yields {\sharp_{\focusA}} A \jtype} \and
    \end{mathpar}
    By definition
    $\focusA(\Gamma, w :_{\focusB} W, \Gamma') \defeq \focusA\Gamma, w
    :_{\focusA\focusB} W, \focusA\Gamma'$, and so we induct on $A$ (now
    weakening with a variable of focus $\focusA\focusB$).
  \end{description}
\end{proof}

\begin{lem}\label{lem:pro-divide-wk}
  For any two focuses $\focusA$ and $\focusB$, the context
  $\focusB \setminus (\focusA \Gamma)$ is an iterated weakening of
  $\focusA(\focusB \setminus \Gamma)$.
\end{lem}
\begin{proof}
  This can be checked variablewise. Given a variable $x :_{\focusC} A$ in
  $\Gamma$, if it survives to $\focusA(\focusB \setminus \Gamma)$ as
  $x :_{\focusA\focusC} A$, then we must have $\focusC \leq
  \focusB$. Multiplying by $\focusA$, it follows
  that $\focusA \focusC \leq \focusA \focusB \leq \focusB$, and so
  $x :_{\focusA\focusC} A$ also occurs in $\focusB \setminus (\focusA \Gamma)$.
\end{proof}

\begin{lem}
  The following equations involving contexts and telescopes hold:
  \begin{align*}
    (\focusA \Gamma')[s/z] &\defeq \focusA (\Gamma'[s/z]) \\
    (\focusA \setminus \Gamma')[s/z] &\defeq \focusA \setminus (\Gamma'[s/z]) \\
  \end{align*}
\end{lem}

\begin{lem}[Substitution]
  \begin{mathpar}
    \inferrule*[left= subst, fraction={-{\,-\,}-}]
    {\focusB \setminus \Gamma \yields s : S \and \Gamma, z :_{\focusB} S, \Gamma' \yields \judge}
    {\Gamma, \Gamma'[s/z] \yields \judge[s/z]} \\
  \end{mathpar}
\end{lem}
\begin{proof}
  Induction on $\judge$.

  \begin{description}[style=unboxed, labelwidth=\linewidth, font=\normalfont\itshape, listparindent=0pt]
  \item[Case (variable).]
    Three subcases as usual, for $x \in \Gamma$, $x \defeq z$ and $x \in
    \Gamma'$. The interesting one is $x \defeq z$:
    \begin{mathpar}
      \inferrule*[left=var]
      {~}
      {\Gamma, z :_{\focusB} S, \Gamma' \yields s : S}
    \end{mathpar}
    Applying \rulen{divide-wk} to $\focusB \setminus \Gamma \yields s : S$ gives
    $\Gamma \yields s : S$, which can be further weakened to
    $\Gamma, \Gamma' \yields s : S$.
  \item[Case (division).]
    \begin{mathpar}
      \inferrule*[left=$\flat$-form]
      {\focusA \setminus (\Gamma, z :_{\focusB} S, \Gamma') \yields A \jtype}
      {\Gamma, s :_{\focusB} S, \Gamma' \yields {\flat_{\focusA}} A \jtype}
    \end{mathpar}
    There are two subcases:
    \begin{itemize}
    \item If $\focusB \leq \focusA$: then
      $\focusA \setminus (\Gamma, z :_{\focusB} S, \Gamma') \defeq (\focusA
      \setminus \Gamma), z :_{\focusB} S, (\focusA \setminus \Gamma')$, in which
      case we induct, getting
      $(\focusA \setminus \Gamma), (\focusA \setminus \Gamma')[s/z] \yields
      A[s/z] \jtype$. This context is equal to
      $\focusA \setminus (\Gamma, \Gamma'[s/z]) \ctx$, so we can reapply the
      rule.
    \item If $\focusB \not\leq \focusA$: then
      $\focusA \setminus (\Gamma, z :_{\focusB} S, \Gamma') \defeq (\focusA
      \setminus \Gamma), (\focusA \setminus \Gamma') \defeq \focusA \setminus
      (\Gamma, \Gamma')$, and so $z$ does not occur in $A$, and
      $A \defeq A[s/z]$, in which case we may reapply the rule.
    \end{itemize}
  \item[Case (promotion).]
    \begin{mathpar}
      \inferrule*[left=$\sharp$-form]
      {\focusA(\Gamma, z :_{\focusB} S, \Gamma') \yields A \jtype}
      {\Gamma, z :_{\focusB} S, \Gamma' \yields {\sharp_{\focusA}} A \jtype} \and
    \end{mathpar}
    By definition
    $\focusA(\Gamma, s :_{\focusB} S, \Gamma') \defeq \focusA \Gamma, s
    :_{\focusA\focusB} S, \focusA\Gamma'$. Applying \rulen{promote} to
    $\focusB \setminus \Gamma \yields s : S$ yields
    $\focusA(\focusB \setminus \Gamma) \yields s : S$. By
    \cref{lem:pro-divide-wk}, this can be weakened to
    $\focusB(\focusA \setminus \Gamma) \yields s : S$, whose context is equal to
    $(\focusA\focusB) \setminus (\focusA \Gamma)$, which is of the correct shape
    to be substituted into
    $\focusA \Gamma, z :_{\focusA\focusB} S, \focusA\Gamma' \yields A
    \jtype$. Substitution gives
    $\focusA \Gamma, (\focusA\Gamma')[s/z] \yields A[s/z] \jtype$, and this
    context is equal to $\focusA (\Gamma, \Gamma'[s/z]) \ctx$, so we can
    reapply the rule.
  \end{description}
\end{proof}

\begin{lem}[Promote]
  \begin{mathpar}
    \inferrule*[left=promote-ctx, fraction={-{\,-\,}-}]
    {\Gamma \ctx}{\focusB\Gamma \ctx} \and
    \inferrule*[left=promote, fraction={-{\,-\,}-}]
    {\Gamma \yields \judge}{\focusB \Gamma \yields \judge} \\
  \end{mathpar}
\end{lem}
\begin{proof}
  First, \rulen{promote-ctx} is by induction on the length of the
  context. Consider a context $\Gamma, x :_{\focusA} A \ctx$, so that
  $\focusA \setminus \Gamma \yields A \jtype$. Applying \rulen{promote} to $A$
  gives $\focusB(\focusA \setminus \Gamma) \yields A \jtype$, which can be
  weakened to $(\focusB\focusA) \setminus (\focusB \Gamma) \yields A \jtype$ by
  \cref{lem:pro-divide-wk}, letting us form the context
  $\focusB \Gamma, x :_{\focusB\focusA} A \ctx$.

  \begin{description}[style=unboxed, labelwidth=\linewidth, font=\normalfont\itshape, listparindent=0pt]
  \item[Case (variable).]  The variable rule is immediate, because modifying the
    annotation on a variable does not change whether it is usable.
  \item[Case (division).]
    \begin{mathpar}
      \inferrule*[left=$\flat$-form]
      {\focusA \setminus \Gamma \yields A \jtype}
      {\Gamma \yields {\flat_{\focusA}} A \jtype}
    \end{mathpar}
    Inductively $\focusB(\focusA \setminus \Gamma) \yields A \jtype$, which can
    be weakened to $\focusA \setminus (\focusB\Gamma) \yields A \jtype$, and we
    reapply the rule.
  \item[Case (promotion).]
    \begin{mathpar}
      \inferrule*[left=$\sharp$-form]
      {\focusA \Gamma \yields A \jtype}
      {\Gamma \yields {\sharp_{\focusA}} A \jtype}
    \end{mathpar}
    Inductively $\focusB(\focusA \Gamma) \yields A \jtype$, and $\focusB(\focusA
    \Gamma) \defeq \focusA(\focusB \Gamma)$, so we may reapply the rule.
  \end{description}
\end{proof}



\begin{lem}[Division]
  \begin{mathpar}
  \inferrule*[left=divide-ctx, fraction={-{\,-\,}-}]{\Gamma \ctx}{\focusB
    \setminus \Gamma \ctx} \and

  \inferrule*[left=divide-wk, fraction={-{\,-\,}-}]{\Gamma \ctx \and \focusB \setminus \Gamma
    \yields \judge}{\Gamma \yields \judge} \and
  \end{mathpar}
\end{lem}
\begin{proof}
  First \rulen{divide-ctx}. Consider a context $\Gamma, x :_{\focusA} A \ctx$, so that
  $\focusA \setminus \Gamma \yields A \jtype$. There are two cases:
  \begin{itemize}
  \item If $\focusA \leq \focusB$: Then \[\focusA \setminus \Gamma \defeq
    (\focusA \focusB) \setminus \Gamma \defeq \focusA \setminus (\focusB
    \setminus \Gamma), \] and so $\focusA \setminus (\focusB
    \setminus \Gamma) \yields A \jtype$ is of the right shape to form $\focusB
    \setminus \Gamma, x :_{\focusA} A \ctx$.
  \item If $\focusA \not\leq \focusB$: Then $\focusB \setminus (\Gamma, x
    :_{\focusA} A) \defeq \focusB \setminus \Gamma$ is well-formed by induction.
  \end{itemize}

  Now \rulen{divide-wk}. On terms:
  \begin{description}[style=unboxed, labelwidth=\linewidth, font=\normalfont\itshape, listparindent=0pt]
  \item[Case (variable).] If $x :_{\focusA} A$ is in context
    $\focusB \setminus \Gamma$, then it must also be in $\Gamma$ and so we may
    reuse the variable rule.
  \item[Case (division).]
    \begin{mathpar}
      \inferrule*[left=$\flat$-form]
      {\focusA \setminus (\focusB \setminus \Gamma) \yields A \jtype}
      {\focusB \setminus \Gamma \yields {\flat_{\focusA}} A \jtype}
    \end{mathpar}
    We know
    $\focusA \setminus (\focusB \setminus \Gamma) \defeq \focusB \setminus
    (\focusA \setminus \Gamma)$, and so inductively
    $\focusA \setminus \Gamma \yields A \jtype$, and we can reapply the rule.
  \item[Case (promotion).]
    \begin{mathpar}
      \inferrule*[left=$\sharp$-form]
      {\focusA (\focusB \setminus \Gamma) \yields A \jtype}
      {\focusB \setminus \Gamma \yields {\sharp_{\focusA}} A \jtype}
    \end{mathpar}
    $\focusA (\focusB \setminus \Gamma) \yields A \jtype$ may be weakened to
    $\focusB \setminus (\focusA \Gamma) \yields A \jtype$ (without increasing
    the size of the derivation). Inductively $\focusA \Gamma \yields A \jtype$,
    and we can reapply the rule.
  \end{description}
\end{proof}

\end{document}